\documentclass[10pt,reqno]{amsart}
\usepackage{amssymb,mathrsfs,graphicx}
\usepackage{ifthen}
\usepackage{hyperref}
\usepackage[margin=1in]{geometry}
\usepackage{caption}
\usepackage{sidecap}
\usepackage{enumitem}
\usepackage{rotating,dsfont}

\usepackage{colortbl}
\definecolor{black}{rgb}{0.0, 0.0, 0.0}
\definecolor{red}{rgb}{1.0, 0.5, 0.5}
\provideboolean{shownotes} 
\setboolean{shownotes}{true} 
\newcommand{\margnote}[1]{
\ifthenelse{\boolean{shownotes}}%
{\marginpar{\raggedright\tiny\texttt{#1}}}%
{}%
}
\newcommand{\hole}[1]{
\ifthenelse{\boolean{shownotes}}%
{\begin{center} \fbox{ \rule {.25cm}{0cm} \rule[-.1cm]{0cm}{.4cm}
\parbox{.85\textwidth}{\begin{center} \texttt{#1}\end{center}} \rule
{.25cm}{0cm}}\end{center}} {} }

\title[Quantitative light-particle limit for the Vlasov--Fokker--Planck--Navier--Stokes system]{Quantitative light-particle limit for the Vlasov--Fokker--Planck--Navier--Stokes system}

\author[Choi]{Young-Pil Choi}
\address[Young-Pil Choi]{\newline Department of Mathematics\newline
Yonsei University, 50 Yonsei-Ro, Seodaemun-Gu, Seoul 03722, Republic of Korea}
\email{ypchoi@yonsei.ac.kr}

\author[Jung]{Jinwook Jung}
\address[Jinwook Jung]{\newline Department of Mathematics and  Research Institute for Natural Sciences \newline
Hanyang University, 222 Wangsimni-ro, Seongdong-gu, Seoul 04763, Republic of Korea}
\email{jinwookjung@hanyang.ac.kr}

\numberwithin{equation}{section}

\newtheorem{theorem}{Theorem}[section]
\newtheorem{lemma}{Lemma}[section]

\newtheorem{proposition}{Proposition}[section]
\newtheorem{remark}{Remark}[section]
\newtheorem{definition}{Definition}[section]

\newcommand{\R}{\mathbb R}

\newcommand{\om}{\Omega}

\newcommand{\N}{\mathbb N}
\newcommand{\ls}{\lesssim}

\newcommand{\T}{\mathbb T}

\newcommand{\bq}{\begin{equation}}
\newcommand{\eq}{\end{equation}}
\newcommand{\e}{\varepsilon}
\newcommand{\lt}{\left}
\newcommand{\rt}{\right}
\newcommand{\lal}{\langle}
\newcommand{\ral}{\rangle}
\newcommand{\pa}{\partial}

\newcommand{\mh}{\mathcal{H}}

\newcommand{\md}{\mathcal{D}}

\newcommand{\intr}{\int_{\R^2}}

\newcommand{\intrr}{\iint_{\om \times \R^2}}

\newcommand{\into}{\int_\Omega}

\newcommand{\calL}{\mathcal L}

\makeatletter
\def\moverlay{\mathpalette\mov@rlay}
\def\mov@rlay#1#2{\leavevmode\vtop{%
   \baselineskip\z@skip \lineskiplimit-\maxdimen
   \ialign{\hfil$\m@th#1##$\hfil\cr#2\crcr}}}
\newcommand{\charfusion}[3][\mathord]{
    #1{\ifx#1\mathop\vphantom{#2}\fi
        \mathpalette\mov@rlay{#2\cr#3}
      }
    \ifx#1\mathop\expandafter\displaylimits\fi}
\makeatother

\begin{document}
\allowdisplaybreaks

\date{\today}


\begin{abstract} 
We investigate the hydrodynamic limit of the Vlasov--Fokker--Planck--Navier--Stokes system in the light particle regime, where the particle relaxation takes place on a singularly fast time scale. Using a relative entropy method adapted to this scaling, we develop the first quantitative convergence theory for the light particle limit. Our analysis yields explicit rates for the convergence of both the kinetic distribution and the fluid velocity, extending the qualitative compactness-based result of Goudon, Jabin, and Vasseur [Indiana Univ. Math. J., 53, (2004), 1495--1515].  Moreover, we derive refined convergence estimates for the macroscopic density and fluid velocity in negative Sobolev spaces, consistent with the formally optimal rates predicted by the Hilbert expansion. The results apply to both the torus and the whole space,  providing a unified quantitative description of the light particle hydrodynamic limit.
\end{abstract}

\maketitle \centerline{\date}

\tableofcontents

%
%
%
%
\section{Introduction}

The purpose of this paper is to investigate the asymptotic analysis of a coupled system describing the evolution of a cloud of dispersed particles immersed in an incompressible viscous fluid. Such particle-fluid interactions arise in a wide range of industrial and scientific applications, including sedimentation processes in medical and chemical engineering, wastewater treatment, as well as the modeling of aerosols and sprays in mechanical systems \cite{BBJM05,Wil58,Wil85}. We refer to \cite{Des10,Mat06,ORo81} and references therein for additional background on the physical mechanisms underlying particle-fluid coupling.
 
At the microscopic level, the dispersed phase is described by its distribution function $f = f(t,x,\xi)$, where $(t,x,\xi) \in \R_+ \times \Omega \times \R^2$ and $\Omega = \T^2$ or $\R^2$. The particle distribution evolves according to a Vlasov--Fokker--Planck equation of the form
\[
\pa_t f + \xi \cdot \nabla_x f = \nabla_\xi \cdot (\nabla_\xi f - F_d f),
\]
where the drag force $F_d$ is taken to be proportional to the relative velocity between the fluid and the particles \cite{Bri49}, 
\[
F_d  = v(t,x) - \xi.
\]
The surrounding fluid is governed by the incompressible Navier--Stokes system
\begin{align*}
&\pa_t v + v \cdot \nabla_x v + \nabla_x P -  \Delta_x v = -\intr F_d f\,d\xi,\cr
&\nabla_x \cdot v = 0.
\end{align*}
 
The mathematical analysis of this kinetic-fluid model is highly nontrivial, largely due to the intrinsic mismatch between the Eulerian description of the fluid variables and the phase-space formulation of the particle distribution. The kinetic density $f(t,x,\xi)$ is posed on the full phase space and evolves along characteristic particle trajectories, whereas the fluid velocity $v(t,x)$ depends solely on the spatial coordinates. This structural asymmetry introduces substantial difficulties in controlling nonlinear couplings, integrating different regularity frameworks, and establishing compactness in mixed variables. These challenges have motivated a long line of research on the analytical properties of kinetic-fluid systems, and a comprehensive theory has gradually emerged over the past decades.
 
From the analytical viewpoint, the well-posedness theory for Vlasov--Navier--Stokes systems has been developed through a substantial sequence of contributions. Global weak solutions in bounded or periodic domains were first obtained in \cite{BDGM09, Ham98}. Subsequent works extended these theories to more complex boundary behaviors such as specular reflection or absorption \cite{BGM17, Yu13}, and to inhomogeneous incompressible backgrounds with large initial data \cite{WY15}. The uniqueness of weak solutions in two dimensions was proved in \cite{DS26, HEM20}. A global weak solution theory for the Vlasov--Fokker--Planck--Navier--Stokes system in the whole space for two and three dimensions was established in \cite{CKL11}, while boundary value problems were studied in \cite{CJ21, LLY22}. For small and smooth perturbations around global Maxwellians, small-data global classical solutions and large-time decay estimates for Vlasov--Fokker--Planck--Navier--Stokes type systems were obtained in \cite{CDM11,DL13,GHMZ10}. In contrast, small-data global classical solutions for the Vlasov--inhomogeneous Navier--Stokes system were derived in \cite{CK15}. Beyond these foundational works, the coupling between Vlasov dynamics and various non-Newtonian or generalized fluid models has been analyzed in \cite{CJWpre,HKKP18,KKK23,MPP18,ZFG24}, together with additional studies concerning collision operators in the kinetic equation \cite{CJ24,CLY21,CY20,GY20,YY18}. Several works have also addressed qualitative behavior and large-time asymptotics of kinetic-fluid solutions \cite{Cho16,Han22,HMM20}. Collectively, these developments provide a robust framework for analyzing Vlasov--Navier--Stokes type systems.
 
In this work, we focus on the asymptotic regime associated with the \emph{light particle scaling}, originally introduced in \cite{GJV04}. Under this scaling, the particle mass is much smaller than that of the fluid, and the relaxation toward equilibrium takes place on an increasingly fast time scale. The resulting singularly perturbed Vlasov--Fokker--Planck--Navier--Stokes system is given by
\begin{align}\label{main1}
\begin{aligned}
&\pa_t f^\e + \frac1\e \xi \cdot \nabla_x f^\e  = \frac1{\e^2}\nabla_\xi \cdot (  \nabla_\xi f^\e + (\xi - \e v^\e ) f^\e), \quad  (t,x,\xi) \in \R_+ \times \Omega \times \R^2,\cr
&\pa_t v^\e + v^\e \cdot \nabla_x v^\e + \nabla_x P^\e -  \Delta_x v^\e = \frac1\e \intr (\xi - \e v^\e)f^\e\,d\xi,\cr
&\nabla_x \cdot v^\e = 0,
\end{aligned}
\end{align}
posed on $\Omega = \T^2$ or $\R^2$. The goal is to identify the effective macroscopic dynamics as $\e \to 0$ and to justify rigorously the corresponding hydrodynamic limit; see \cite{GJV04} for the derivation and dimensional analysis of this scaling.
 
The asymptotic analysis of Vlasov--Navier--Stokes and Vlasov--Fokker--Planck--Navier--Stokes systems has been extensively studied over the past two decades, with two distinct scaling regimes, light-particle (parabolic) and fine-particle (hyperbolic), leading to markedly different analytical structures. See \cite{CG06, CGL08} for related asymptotic regimes and methodologies.

In the light-particle regime, the first convergence result was obtained for the Vlasov--Fokker--Planck / incompressible Navier--Stokes system on $\T^2\times\R^2$ using compactness arguments \cite{GJV04}. Further extensions to inhomogeneous incompressible Navier--Stokes flows in bounded domains were also derived through qualitative compactness methods \cite{Gha26}, and related one-dimensional Vlasov--compressible Navier--Stokes couplings were treated similarly \cite{CWY20}. These analyses provide qualitative convergence but no quantitative bounds, and they are restricted to periodic or bounded spatial domains.

In contrast, the fine-particle (hyperbolic) regime has been successfully treated by the relative entropy method, which yields quantitative stability and convergence estimates. This includes results for Vlasov--Fokker--Planck / incompressible Navier--Stokes systems in the whole space \cite{GJV04_2}, for compressible Navier--Stokes couplings in bounded domains \cite{MV08}, and for Vlasov or Vlasov--Fokker--Planck interactions with inhomogeneous incompressible Navier--Stokes flows \cite{GM23, SY20, Gha10, SWYZ23}.  Further quantitative studies for nonlinear Fokker--Planck / Navier--Stokes systems appear in \cite{CJ21, CJ23}, and modulated energy methods capable of treating both scaling regimes for Vlasov / incompressible Navier--Stokes system in periodic domains are developed in \cite{HM24}.  We also refer to the earlier qualitative analysis of the hyperbolic regime in 1D for a Vlasov / viscous Burgers coupling \cite{Gou01}.

Despite the extensive literature on kinetic-fluid limits, quantitative results for the light-particle regime have not been available in any spatial setting. Existing works in this scaling \cite{CWY20, Gha26, GJV04} are qualitative in nature, relying on compactness arguments, and thus provide no convergence rates for either the kinetic or fluid components. Moreover, these analyses are restricted to periodic or bounded domains, and even qualitative convergence in the whole space has not been established.

The present work addresses these limitations by developing the \emph{first quantitative stability framework} for the light-particle limit, applicable both on the torus and in the whole space. Our approach adapts the unified relative entropy method introduced in \cite{CJpre} to the singular light-particle scaling, yielding explicit convergence rates for the kinetic and fluid variables and, in particular, refined second-order estimates for the macroscopic density and fluid velocity in negative Sobolev spaces. This quantitative framework fills the gap left by previous qualitative studies and forms the core of the main theorem stated below.

%
%
%
%
%
 
\subsection{Formal asymptotic limit $\e \to 0$} We outline here the formal derivation of the limiting system as $\e \to 0$. First, the kinetic equation in \eqref{main1} can be rewritten as
\[
\e \pa_t f^\e + \xi \cdot \nabla_x f^\e = \frac1\e \nabla_\xi \cdot \lt(f^\e \nabla_\xi \log\lt(\frac{f^\e}{M_{\rho_{f^\e}, \e v^\e}} \rt) \rt),
\]
where
\[
M_{\rho_{f^\e}, \e v^\e} := \frac{\rho_{f^\e}}{2\pi} \exp\lt( - \frac{|\xi - \e v^\e|^2}{2} \rt).
\]
This formulation shows that, as $\e \to 0$, the right-hand side enforces a fast relaxation of $f^\e$ toward the local Maxwellian $M_{\rho_{f^\e}, \e v^\e}$. At the formal level, the drag forcing term in the fluid equation converges to $-\nabla_x\rho$, which is a pure gradient and can thus be absorbed into the pressure. Hence, $v^\e$ converges to the solution of the decoupled incompressible Navier--Stokes system:
 \begin{align*}
&\pa_t v + v\cdot \nabla_x v + \nabla_x P -  \Delta_x v = 0,\cr
&\nabla_x \cdot v = 0.
\end{align*}

For the kinetic component, computing the local mass and momentum balances gives
 \begin{align*}
&\e \pa_t \rho_{f^\e} + \nabla_x \cdot m_{f^\e} = 0, \cr
&\e \pa_t m_{f^\e} + \nabla_x \cdot \lt(\intr \xi \otimes \xi f^\e\,d\xi \rt) = -\frac1\e m_{f^\e} + \rho_{f^\e} v^\e, 
\end{align*}
where  
\[
m_f: = \intr \xi f\,d\xi.
\]
Combining the above yields
\bq\label{formal_d_kin}
\pa_t \rho_{f^\e} + \nabla_x \cdot (\rho_{f^\e} v^\e) =\Delta_x \rho_{f^\e} + \e \nabla_x \cdot \pa_t m_{f^\e} + \nabla_x \otimes \nabla_x : \intr (\xi \otimes \xi - \mathbb{I}_2) f^\e \,d\xi.
\eq
Since $f^\e$ relaxes to $M_{\rho_{f^\e}, \e v^\e}$, we have
\[
\intr \xi \otimes \xi M_{\rho_{f^\e}, \e v^\e}\,d\xi = \rho_{f^\e} \lt(  \e^2 v^\e \otimes v^\e +  \mathbb{I}_2  \rt)  \quad \mbox{and} \quad   \intr \xi M_{\rho_{f^\e}, \e v^\e} \,d\xi = \e\rho_{f^\e} v^\e. 
\]
Thus, we obtain the approximation 
\begin{align*}
&\e \nabla_x \cdot \pa_t m_{f^\e} + \nabla_x \otimes \nabla_x : \intr (\xi \otimes \xi - \mathbb{I}_2) f^\e\,d\xi  \cr
&\quad \sim \e^2\lt[\nabla_x \cdot \pa_t (\rho_{f^\e} v^\e) + \nabla_x \otimes \nabla_x :  (\rho_{f^\e} v^\e \otimes v^\e) \rt]= \mathcal{O}(\e^2), \quad (\e \ll 1).
\end{align*}
Hence, in the limit $\e\to0$, the right-hand side of \eqref{formal_d_kin} converges formally to 0, yielding the limiting advection-diffusion equation
\[
\pa_t \rho +  \nabla_x \cdot (\rho v) =   \Delta_x \rho.
\]

In summary, the formal limit of \eqref{main1} is the coupled advection-diffusion/incompressible Navier--Stokes system:
\begin{align}\label{lim_light}
\begin{aligned}
&\pa_t \rho +  \nabla_x \cdot (\rho v) =   \Delta_x \rho, \cr
&\pa_t v + v\cdot \nabla_x v + \nabla_x P -  \Delta_x v = 0,\cr
&\nabla_x \cdot v = 0.
\end{aligned}
\end{align}

%
%
%
%
%

\subsection{Main result} We now present the precise formulation of our quantitative convergence result for the light-particle scaling. Before stating the theorem, we introduce the analytical tools and notational conventions that will be used throughout the discussion. 

Our approach relies fundamentally on a relative entropy structure adapted to the Vlasov--Fokker--Planck dynamics. For a nonnegative distribution $f: \R_+ \times \Omega \times \R^2 \to \R_+$, and macroscopic fields $\rho: \R_+ \times \Omega \to \R_+$, $u:\R_+ \times \Omega \to \R^2$, we define the relative entropy with respect to the local Maxwellian
\[
M_{\rho, u}(\xi) := \frac{\rho}{2\pi }\exp\lt( -\frac{|u-\xi|^2}{2} \rt)
\]
by
\[
\mathscr{H}[f | M_{\rho,u}] := \intrr f \log \frac{f}{M_{\rho,u}}  \, dx d\xi - \intrr \lt(f - M_{\rho,u} \rt) dx d\xi.
\]
Throughout the paper we assume, without loss of generality, that $f$ has unit mass:
\[
\iint_{\om \times \R^2} f(t,x,\xi)\,dxd\xi = 1, \quad t \geq 0.
\]
This normalization is preserved by the kinetic equation and simplifies several entropy expressions appearing later. This structure provides a natural way to measure the discrepancy between the kinetic solution and its Maxwellian approximation, and it plays a central role in quantifying the macroscopic limit.

To capture the refined convergence of the macroscopic density and fluid velocity, we work with negative Sobolev norms. More precisely, for $s\ge0$, we denote by $H^{-s}(\Omega)$ the dual space of $H^s(\Omega)$. These spaces provide a convenient weak topology in which the parabolic structure of the limiting system can be exploited to obtain quantitative estimates for $\rho_{f^\e}-\rho$ and $v^\e-v$. In particular, they are well suited to the optimal second-order expansions derived formally in Appendix \ref{app:op_rate}.

Finally, we record the basic notational conventions used throughout the paper. For functions $f(x,\xi)$ on $\Omega \times \R^2$ and $g(x)$ on $\Omega$, we write  $\|f\|_{L^p}$ and $\|g\|_{L^p}$ for their respective $L^p$-norms. We also employ the velocity-weighted $L^1$-norm defined by 
 \[
 \|f\|_{L^1_2} := \iint_{\om \times \R^2}  (1+|\xi|^2)f\,dxd\xi. 
 \]
 We denote by $\mathcal{P}(\Omega)$ the space of probability measures on $\Omega$. Generic constants are denoted by $C>0$ and may vary from line to line, and we write $f \ls g$ when $f \leq Cg$. For readability, we occasionally omit the $x$-subscript in differential operators and write, for instance, $\nabla f = \nabla_x f$ and $\Delta f = \Delta_x f$.

We now recall the notion of weak entropy solutions to the kinetic-fluid system \eqref{main1}.
This concept, originating from \cite{GJV04}, provides a natural framework for solutions with finite energy and entropy, and will serve as the solution class in which our convergence theorem is formulated.
\begin{definition}\label{def_weak}Let $T>0$. We say that $(f^\e, v^\e)$ is a weak entropy solution to \eqref{main1} with initial data $(f_0^\e, v^\e_0)$ if the following conditions hold:
\begin{enumerate}[label=(\roman*)]
\item $f^\e \in L^\infty([0,T]; L_+^1(\Omega \times \R^2, (1 + |x| + |\xi|^2)\,dxd\xi) \cap L \log L(\Omega \times \R^2))$,
\item $v^\e \in L^\infty([0,T]; L^2(\om)) \cap L^2([0,T];H^1(\om))$,
\item $(f^\e, v^\e)$ satisfies \eqref{main1} in the sense of distributions,
\item the total energy inequality holds for almost every $t \in [0,T]$:
\begin{align*}
& \intrr \lt(\frac12|\xi|^2 + \log f^\e(t)\rt)f^\e(t)\,dxd\xi + \frac12 \|v^\e(t)\|_{L^2}^2 \cr
&\quad +  \frac1{\e^2}\int_0^t\iint_{\om \times \R^2}  \frac1{f^\e}|  \nabla_\xi f^\e + (\xi - \e v^\e) f^\e|^2 \,dxd\xi ds  + \int_0^t\int_\om |\nabla v^\e|^2\,dxds\cr
&\quad \leq  \intrr \lt(\frac12|\xi|^2 + \log f^\e_0\rt)f^\e_0\,dxd\xi + \frac12 \|v^\e_0\|_{L^2}^2 .
\end{align*}
\end{enumerate}
\end{definition}

This class of solutions was introduced in \cite{GJV04} for the periodic domain $\Omega = \T^2$ and later extended to the whole-space setting in \cite{CKL11}. A related entropy framework has also been used in bounded domains for the Vlasov--Fokker--Planck equation coupled with the inhomogeneous Navier--Stokes system \cite{Gha26}. Since the existence theory for weak entropy solutions to \eqref{main1} is already well established in these works, we do not revisit it here.

With these preparations in place, we are now ready to state our quantitative convergence theorem.
 
\begin{theorem}\label{thm_main} Let $T>0$, and let $(\rho,v)$ be the unique classical solution to the system \eqref{lim_light} satisfying
\[
\rho\in L^\infty(0,T; \mathcal{P}\cap L^\infty(\om)), \quad v, \nabla \log \rho \in L^\infty(0,T; W^{1,\infty}(\Omega)) \cap W^{1,\infty}(0,T; L^\infty(\Omega)).
\]
Let $\{(f^\e, v^\e)\}_{\e > 0}$ be a family of weak entropy solutions to the equation \eqref{main1} on the time interval $[0,T]$, with initial data $\{(f^\e_0, v^\e_0)\}_{\e > 0}$ satisfying the uniform bounds
\[
\sup_{\e > 0}\lt( \||x|f_0^\e\|_{L^1}+ \|f^\e_0\|_{L^1_2 \cap L \log L} +  \|v^\e_0\|_{L^2}\rt) < \infty.
\]
If the initial data is well-prepared in the sense that
\[ 
  \mathscr{H}[f^\e | M_{\rho, 0}](0) + \|v^\e_0 - v_0\|_{L^2}^2 \to 0 \quad \mbox{as } \e \to 0,
\] 
then we have
\begin{align*}
f^\e &\to M_{\rho, 0} \quad \mbox{in } L^\infty(0,T; L^1(\Omega \times \R^2)), \cr
  \intr f^\e\,d\xi &\to \rho \quad \mbox{in } L^\infty(0,T; L^1(\Omega)),\cr
  \frac1\e \intr \xi f^\e\,d\xi  &\to \rho\lt(v- \nabla \log \rho \rt) \quad \mbox{in } L^1((0,T) \times \Omega),\cr
  v^\e &\to v \quad \mbox{in } L^\infty(0,T; L^2(\Omega)) \cap L^2(0,T;H^1(\om)).
\end{align*}
In addition, the following stability estimate holds:
\begin{align*}
&\|f^\e - M_{\rho, 0}\|_{L^\infty(0,T;L^1)}^2 + \|\rho_{f^\e} - \rho\|_{L^\infty(0,T;L^1)}^2 + \|v^\e - v\|_{L^\infty(0,T;L^2) \cap L^2(0,T; H^1)}^2 \cr
&\quad +  \|  \frac1\e \intr \xi f^\e\,d\xi  - \rho\lt(v- \nabla \log \rho \rt)\|_{L^1((0,T) \times \Omega)}^2  \cr
 &\quad \leq C\lt(  \mathscr{H}[f^\e | M_{\rho, 0}](0) + \|v^\e_0 - v_0\|_{L^2}^2\rt) + C\e^2.
\end{align*}

 If we further assume that $v \in L^\infty(0,T;H^{s_1+1}(\Omega))$ and $\rho \in L^\infty(0,T; H^{s_1}(\Omega))$  for some $s_1 > 1$, then the density and fluid velocity errors satisfy the following refined quantitative estimates in negative Sobolev spaces:
 \begin{align}\label{ess_opt}
 \begin{aligned}
 \|v^\e-v\|_{L^2(0,T;H^{-s_1})} &\le C\lt( \|v_0^\e - v_0\|_{H^{-s_1}} + \mathscr{H}[f^\e | M_{\rho, 0}](0)  + \|v^\e_0 - v_0\|_{L^2}^2 + \e^2\rt), \\
 \|\rho_{f^\e} -\rho\|_{L^2(0,T;H^{-s_1-1})}  &\le C\lt(\|\rho_0^\e - \rho_0\|_{H^{-s_1-1}}  + \|v_0^\e - v_0\|_{H^{-s_1}} +  \mathscr{H}[f^\e |M_{\rho, 0}](0) + \|v_0^\e - v_0\|_{L^2}^2 + \e^2 \rt).
 \end{aligned}
 \end{align}
Here, the constant $C>0$ is independent of $\e>0$, depending only on the initial data, $T$, and the limiting solution $(\rho, v)$.

\end{theorem}
 
 \begin{remark}\label{rmk_expan} 
The quantitative convergence of all macroscopic quantities follows directly from the structure of the modulated energy.  If we write $u_f := \frac{m_f}{\rho_f}$, the algebraic decomposition
\bq\label{eq_exp}
  \intr f \log \frac{f}{M_{\rho, u}}\,d\xi = \intr f \log \frac{f}{M_{\rho_f, u_f}}\,d\xi +   \rho_f \log \frac{\rho_f}{\rho} + \frac{\rho_f}{2}|u_f - u|^2,
\eq
which is valid for any $\rho \in \R_+$ and $u \in \R^2$, separates the kinetic, density, and
momentum contributions in the relative entropy. Controlling the relative entropy $\mathscr{H}[f^\e |M_{\rho,0}]$ therefore yields the quantitative $L^1$-convergence of $f^\e$ to $M_{\rho,0}$ and of $\rho_{f^\e}$ to $\rho$ through the Csisz\'ar--Kullback--Pinsker inequality. The same entropy dissipation also supplies a quantitative bound for the scaled momentum density $\frac1\e m_{f^\e}$ since its coercive part appears explicitly in the relative entropy dissipation. 

The convergence of the fluid velocity $v^\e$ is obtained from the fluid component of the modulated energy, which is tightly coupled with the kinetic entropy. In particular, the decay of the total modulated energy directly yields the quantitative convergence $v^\e \to v$ in the norms stated in Theorem \ref{thm_main}.
 \end{remark}

  \begin{remark} 
The stability estimates in Theorem \ref{thm_main} are consistent with the formal expansion
derived in Appendix \ref{app:op_rate}. Accordingly, the term \emph{optimal} should be understood here in a formal and variable-dependent sense, namely with respect to the asymptotic orders predicted by the Hilbert expansion and in the topology in which the estimate is obtained.

For the kinetic density $f^\e$, the estimate
\[
\|f^\e-M_{\rho,0}\|_{L^\infty(0,T;L^1)}^2 \ls \mathscr{H}[f^\e | M_{\rho, 0}](0)+\|v^\e_0-v_0\|_{L^2}^2 + \e^2
\]
is optimal at the level of the formal expansion. Indeed, the formal Hilbert expansion shows that
\[
f^\e= M_{\rho,0} + \e f_1 + O(\e^2),
\]
with a nontrivial first-order corrector $f_1$.

For the macroscopic density $\rho_{f^\e}$, the first-order correction cancels at the formal level, and Appendix \ref{app:op_rate} predicts the second-order expansion
\[
\rho_{f^\e}=\rho+O(\e^2).
\]
Accordingly, the negative Sobolev estimate in \eqref{ess_opt},
\[
\|\rho_{f^\e}-\rho\|_{L^2(0,T;H^{-s_1-1})}\ls \e^2
\]
under suitably well-prepared initial data of order $O(\e^2)$, is consistent with the formal asymptotics and may therefore be regarded as essentially optimal in this weak topology.

A similar phenomenon occurs for the fluid velocity. While the strong estimate in Theorem \ref{thm_main} yields only an $O(\e)$ convergence rate in the natural energy norm, the refined estimate in \eqref{ess_opt} gives
\[
\|v^\e-v\|_{L^2(0,T;H^{-s_1})}\ls \e^2.
\]
This is again consistent with Appendix \ref{app:op_rate}, where the forcing term in the fluid equation exhibits a second-order cancellation and, under well-prepared initial data, the formal expansion yields
\[
v^\e-v=O(\e^2).
\]
Hence, this rate should also be viewed as formally optimal in the negative Sobolev topology.

In this sense, the strong norms in Theorem \ref{thm_main} capture the first-order convergence of the full kinetic-fluid system, whereas the negative Sobolev estimates in \eqref{ess_opt} recover the essentially optimal second-order behavior of the macroscopic density and fluid velocity predicted by the formal asymptotics.

\end{remark}

\begin{remark}
The estimate for $v^\e-v$ in \eqref{ess_opt} only requires the bound for $\|v\|_{L^\infty(0,T;H^{s_1})}$, while the estimate for $\rho_{f^\e}-\rho$ requires one additional derivative. For this reason, Theorem \ref{thm_main} is stated under the stronger assumption $v\in L^\infty(0,T;H^{s_1+1}(\Omega))$. See Section \ref{sec_weak-conv} for the detailed derivation of these negative Sobolev estimates.
\end{remark}

%
%
%
%
%

 \subsection{Organization of the paper}

The remainder of the paper is organized as follows. Section \ref{sec_pre} collects several analytical ingredients used throughout the proof. We first establish a weighted Trudinger--Moser inequality, which provides the additional integrability needed to control mixed kinetic--fluid interaction terms. We then derive a family of uniform a priori estimates, including moment and entropy bounds, that will be used repeatedly in the relative entropy analysis. Section \ref{sec_stab} is devoted to the relative entropy structure of the light-particle limit. We reformulate the limiting system in a suitable relaxation form and derive the key relative entropy inequality comparing the kinetic--fluid system with the limiting advection--diffusion / incompressible Navier--Stokes system. In Section \ref{sec_proof}, we complete the proof of Theorem \ref{thm_main}. We first close the relative entropy inequality and obtain the quantitative convergence estimates in the well-prepared regime. We then derive refined convergence rates for the fluid velocity and density in negative Sobolev spaces. Finally, the appendices address complementary aspects of the asymptotic analysis. Appendix \ref{app:op_rate} presents the formal diffusion limit and identifies the expansion rates suggested by a Hilbert expansion. Appendix \ref{app:ext} establishes the existence theory for the limiting system, including a continuation criterion and global-in-time existence on the torus, thereby justifying the regularity assumptions imposed in Theorem \ref{thm_main}.

%
%
%
%

\section{Preliminaries}\label{sec_pre}

%
%
%
%
%
\subsection{Weighted Trudinger--Moser inequality}

In this subsection, we establish a weighted Trudinger--Moser inequality that will play a central role throughout the relative entropy analysis. The drag force induces a strong coupling between the kinetic and fluid components, and several nonlinear interaction terms arise in which the particle density $\rho_{f^\e}$ is controlled only in $L\log L$, whereas the fluid velocity $v^\e$ possesses $H^1$ regularity.  Such mixed products cannot be handled by standard Sobolev embeddings in two dimensions, and closing the entropy estimates requires an additional gain of integrability.  The weighted Trudinger--Moser inequality provides exactly this mechanism by restoring the critical integrability needed to control expressions of the form $\rho_{f^\e} |v^\e|^2$ and similar coupling
terms.

This phenomenon is intrinsically two-dimensional: the Trudinger--Moser inequality compensates for the failure of the embedding $H^1(\Omega) \hookrightarrow L^\infty(\Omega)$ in dimension two, and it is precisely at this point that the restriction $\Omega\subset\R^2$ enters our analysis.  In the periodic setting, a related argument appeared implicitly in \cite{GJV04}; however, extending the analysis to the whole space requires a refined, weighted version of the inequality that will allow us to propagate spatial moments and control the entropy production uniformly in time.  

For completeness, and since our framework treats both the periodic and whole-space domains, we now state the weighted Trudinger--Moser inequality in a form suitable for our subsequent estimates.

 \begin{lemma}\label{lem:TM}
Let $T>0$, $\Omega =\T^2$ or $\R^2$, and let $p\in L^\infty([0,T]; \mathcal{P}(\om) \cap L\log L(\om))$ and $\phi\in L^2(0,T;H^1(\Omega))$.  Then there exists a constant $C>0$, depending only on $\Omega$, such that
\[ 
\int_0^T  \int_{\Omega} p |\phi|^2\,dxdt \leq C\lt(1+\sup_{0<t<T}\int_{\Omega} p |\ln p|\,dx\rt) 
\|\phi\|_{L^2(0,T;H^1(\Omega))}^2.
\]
\end{lemma}

\begin{proof}
Fix $t\in(0,T)$ and suppress $t$ in the notation. We argue pointwise in time and then integrate.

\medskip
\noindent\textbf{Step 1 (Fenchel--Young with an exponential Orlicz pair).}
Let $\Phi(\eta):=e^\eta-1$ for $\eta\in\R$. Then its convex conjugate is
\[
\Phi^*(\xi)=\sup_{\eta\in\R}\,(\xi\eta-\Phi(\eta))=(\xi\ln\xi-\xi)+1\quad\text{for }\xi\ge0,
\]
with the convention $0\ln 0=0$. Hence, by Fenchel--Young inequality, for all $\xi\ge0$ and $\eta\in\R$, we obtain
\[
\xi\eta  \le  (e^\eta-1)+\xi\ln\xi-\xi+1.
\]
Fix $\beta>0$ and apply this pointwise with $\xi=p(x)$ and $\eta=\beta\,\psi(x)^2$ (so $\eta\ge0$). We obtain
\[
p \beta \psi^2  \le  (e^{\beta\psi^2}-1)+p\ln p - p + 1,
\]
and thus
\bq\label{eq:Fen}
p \psi^2
 \le \frac{1}{\beta}\lt(e^{\beta\psi^2}-1\rt) + \frac{1}{\beta}\,p\ln p + \frac{1}{\beta}\,(1-p).
\eq

Two consequences will be crucial below:
\begin{itemize}
\item On sets where $p>1$ we have $1-p\le0$, so the last term in \eqref{eq:Fen} can be dropped and we get
\bq\label{eq:Fen1}
p \psi^2 \le \frac{1}{\beta}\lt(e^{\beta\psi^2}-1\rt) + \frac{1}{\beta} p\ln p
\quad\text{on }\{p>1\}.
\eq
\item On sets where $p\le1$ we simply use $p \psi^2\le \psi^2$, avoiding any constant term altogether.
\end{itemize}

\medskip
\noindent\textbf{Step 2 (Choice of $\psi$ and Trudinger--Moser).}
We distinguish the domain.

\smallskip
\emph{Case A: $\Omega=\T^2$.}
Set $\psi=\phi^\circ:=\phi-\bar\phi$, where 
\[
\bar\phi:=\frac1{|\T^2|}\int_{\T^2}\phi\,dx.
\]
Then we get 
\[
\int_{\T^2}\phi^\circ=0, \quad \|\phi^\circ\|_{L^2}\lesssim\|\nabla\phi\|_{L^2}, \quad \mbox{and} \quad \|\nabla\phi^\circ\|_{L^2}=\|\nabla\phi\|_{L^2}.
\]
By the mean-zero Trudinger--Moser inequality on $\T^2$ \cite{GT83},  there exist $\alpha_0>0$ and $K_{\T^2}>0$ such that
\[
\int_{\T^2}\exp \lt(\alpha_0\,\frac{|\phi^\circ|^2}{\|\nabla\phi\|_{L^2}^2}\rt) dx  \le  K_{\T^2}.
\]
If $\|\nabla\phi\|_{L^2}=0$, then $\phi$ is constant and we handle it at the end of the torus case. Otherwise set
\[
\beta:=\frac{\alpha_0}{\|\nabla\phi\|_{L^2}^2}.
\]
Then, we have
\[
\int_{\T^2}\lt(e^{\beta|\psi|^2}-1\rt) dx
=\int_{\T^2}\lt(\exp \lt(\alpha_0\,\frac{|\phi^\circ|^2}{\|\nabla\phi\|_{L^2}^2}\rt)-1\rt) dx
 \le  K_{\T^2}-|\T^2|.
\]

\smallskip
\emph{Case B: $\Omega=\R^2$.}
Set $\psi=\phi$. By the sharp whole-space Trudinger--Moser inequality (e.g. \cite{LR08, R05}), there exist $\alpha_*\in(0,4\pi)$ and $K_{\R^2}>0$ such that, whenever $\|g\|_{H^1(\R^2)}\le1$,
\[
\intr\lt(e^{\alpha_*|g|^2}-1\rt) dx  \le  K_{\R^2}.
\]
If $\|\phi\|_{H^1(\R^2)}=0$, we are done. Otherwise set 
\[
g:=\frac\phi{\|\phi\|_{H^1}} \quad \mbox{and} \quad \beta:=\frac{\alpha_*}{\|\phi\|_{H^1}^2}. 
\]
Then, we have
\bq\label{eq:TM-whole}
\intr\lt(e^{\beta|\psi|^2}-1\rt) dx = \intr\lt(e^{\alpha_*|g|^2}-1\rt) dx  \le  K_{\R^2}.
\eq

\medskip
\noindent\textbf{Step 3 (Estimate on $\T^2$).}
Integrate \eqref{eq:Fen} over $\T^2$ with $\psi=\phi^\circ$ and $\beta=\frac{\alpha_0}{\|\nabla\phi\|_{L^2}^2}$. Using the bound for the exponential term, we obtain
\[
\int_{\T^2} p |\phi^\circ|^2\,dx \le \frac{K_{\T^2}-|\T^2|}{\beta}+\frac{1}{\beta}\int_{\T^2}p|\ln p|\,dx+\frac{1}{\beta}\int_{\T^2}(1-p)\,dx \leq \frac{K_{\T^2}}{\beta}+\frac{1}{\beta}\int_{\T^2}p|\ln p|\,dx.
\]
Thus, we find
\[
\int_{\T^2}p|\phi^\circ|^2\,dx \le \frac{C}{\alpha_0} \|\nabla\phi\|_{L^2}^2+\frac{\|\nabla\phi\|_{L^2}^2}{\alpha_0}\int_{\T^2}p|\ln p|\,dx = C\lt(1+\int_{\T^2}p|\ln p|\,dx\rt) \|\nabla \phi\|_{L^2(\T^2)}^2.
\]
On the other hand, we observe
\[
\int_{\T^2}p |\bar\phi|^2\,dx = |\bar\phi|^2 \le C\,\|\phi\|_{L^2}^2,
\]
and thus
\[
\int_{\T^2}p|\phi|^2\,dx
 \le C\lt(1+\int_{\T^2}p|\ln p|\,dx\rt) \|\phi\|_{H^1(\T^2)}^2.
\]
If $\|\nabla\phi\|_{L^2}=0$ (i.e.\ $\phi$ is constant), the same bound follows from the estimate on the average term. Integrating over $(0,T)$ gives the desired result.

\medskip
\noindent\textbf{Step 4 (Estimate on $\R^2$).}
Split $\R^2=\{p>1\}\cup\{p\le1\}$. On $\{p>1\}$, apply \eqref{eq:Fen1} with $\psi=\phi$ and $\beta=\frac{\alpha_*}{\|\phi\|_{H^1}^2}$, and use \eqref{eq:TM-whole}:
\[
\int_{\{p>1\}}p |\phi|^2\,dx \le \frac{1}{\beta}\intr\lt(e^{\beta|\phi|^2}-1\rt) dx
+\frac{1}{\beta}\intr p|\ln p|\,dx
 \le 
\frac{K_{\R^2}}{\alpha_*}\,\|\phi\|_{H^1}^2
+\frac{\|\phi\|_{H^1}^2}{\alpha_*}\intr p|\ln p|\,dx.
\]
On $\{p\le1\}$, simply, we observe
\[
\int_{\{p\le1\}}p |\phi|^2\,dx\le \intr|\phi|^2\,dx \le \|\phi\|_{H^1}^2. 
\]
Combining, we arrive at
\[
\intr p\,|\phi|^2\,dx \le C\lt(1+\intr p|\ln p|\,dx\rt) \|\phi\|_{H^1(\R^2)}^2,
\]
and integrating it over $(0,T)$ completes the proof.
\end{proof}

%
%
%
%
%
\subsection{Uniform bound estimates}
In this part, we collect several a priori estimates for the system \eqref{main1}. The computations below are carried out at a formal level, but each estimate can be rigorously justified under the classical approximation framework, see, e.g. \cite{CKL11, GJV04, Ham98}. Since our analysis also requires uniform control in the whole-space setting, beyond the periodic domain treated in \cite{GJV04}, we provide the derivations for completeness while referring to the above works for the full justification.

A key quantity is the free energy functional
\[
\mathscr{F}[f,v] :=  \iint_{\om \times \R^2} \lt(\frac{|\xi|^2}2 + \log f \rt) f\,dxd\xi  + \int_\om \frac{|v|^2}{2}\,dx,
\]
whose dissipation plays a central role in the compactness and stability analysis that follows. The next lemma records its basic monotonicity along smooth solutions to \eqref{main1}. This identity corresponds precisely to the energy inequality used in the definition of weak entropy solutions (Definition \ref{def_weak}), but is stated here in the smooth setting to keep the formal computations clear. 

From now on, we write $\rho^\e := \rho_{f^\e}$  and $u^\e := u_{f^\e}$ for simplicity.

\begin{lemma}\label{lem:energy} Let $T>0$, and let $(f^\e, v^\e)$ be sufficiently regular solutions to the equation \eqref{main1} on the time interval $[0,T]$. Then, we have
\begin{align*}
& \mathscr{F}[f^\e(t),v^\e(t)] + \frac1{\e^2}\int_0^t \iint_{\om \times \R^2} \frac1{f^\e}|  \nabla_\xi f^\e + (\xi - \e v^\e) f^\e|^2\,dxd\xi ds  + \int_0^t \int_\om |\nabla v^\e|^2\,dx ds =  \mathscr{F}[f^\e_0,v^\e_0],
\end{align*}
for $t \in [0,T]$.
\end{lemma}
\begin{proof} Differentiating the kinetic contribution yields
\begin{align*}
&\frac{d}{dt}\iint_{\om \times \R^2} \lt(\frac{|\xi|^2}2 + \log f^\e\rt) f^\e\,dxd\xi + \frac1{\e^2}\iint_{\om \times \R^2} \frac1{f^\e}|  \nabla_\xi f^\e + (\xi - \e v^\e) f^\e|^2\,dxd\xi \cr
&\quad =  \frac1\e \iint_{\om \times \R^2} v^\e \cdot (\e v^\e - \xi) f^\e\,dxd\xi. 
\end{align*}
A direct computation from the fluid equation gives
\[
\frac{d}{dt} \int_\om \frac{|v^\e|^2}{2}\,dx +  \int_\om |\nabla v^\e|^2\,dx =  \frac1\e \iint_{\om \times \R^2} v^\e \cdot (\xi - \e v^\e)f^\e\,dxd\xi .
\]
Summing both identities and integrating in time concludes the desired result.
\end{proof}
\begin{remark}\label{rmk_diss} The dissipation term naturally decomposes into a kinetic relaxation part and a hydrodynamic alignment part:
\[
\iint_{\om \times \R^2} \frac1{f^\e}|  \nabla_\xi f^\e + (\xi - \e v^\e) f^\e|^2\,dxd\xi  = \iint_{\om \times \R^2} \frac1{f^\e}|  \nabla_\xi f^\e + (\xi - u^\e) f^\e|^2\,dxd\xi  + \int_\om \rho^\e |u^\e - \e v^\e|^2\,dx.
\]
The first term measures the deviation of $f^\e$ from the local Maxwellian $M_{\rho^\e,u^\e}$, while the second controls the mismatch between the kinetic momentum $u^\e$ and the scaled fluid velocity $\e v^\e$.
\end{remark}

The next step is to obtain uniform control of the spatial moments of $f^\e$ and of the $L\log L$-entropy of the density $\rho^\e$. Such bounds do not appear in the periodic setting treated in \cite{GJV04}, but they become essential in the whole-space regime considered here. In particular, the control of $\intrr |x| f^\e\,dxd\xi$ compensates for the absence of confinement and replaces the compactness mechanisms that are automatically available on $\T^2$.

A further difficulty specific to our setting arises from the coupling term $\int_0^T \int_\Omega \rho^\e |v^\e|^2\,dxdt$, which must be controlled uniformly in $\e$. The weighted Trudinger--Moser inequality established in the previous subsection plays a crucial role here: it provides the necessary integrability to handle mixed kinetic-fluid terms that cannot be controlled by standard Sobolev embeddings in two dimensions.

We now derive uniform bounds on the first moment of $f^\e$ and on the $L\log L$-entropy of $\rho^\e$.

 \begin{lemma} \label{lem:mom} Let $T>0$, and let $(f^\e, v^\e)$ be the weak entropy solution to \eqref{main1} on the time interval $[0,T]$ in the sense of Definition \ref{def_weak}. Then, we have
\[
\sup_{0 \leq t \leq T}\iint_{\om \times \R^2} |x| f^\e(t)\,dxd\xi + \sup_{0 \leq t \leq T} \int_\Omega \rho^\e |\log \rho^\e|\,dx \leq C  \lt(1 + \|xf^\e_0\|_{L^1} + \mathscr{F}[f^\e_0, v^\e_0]\rt),
\]
where $C>0$ depends only on $T$.
\end{lemma}
\begin{proof} We begin by differentiating the first spatial moment. Using the kinetic equation in \eqref{main1} yields
\begin{align*}
\frac{d}{dt}\iint_{\om \times \R^2} |x| f^\e\,dxd\xi &= \frac1\e \iint_{\om \times \R^2} \frac{x \cdot \xi}{|x|}f^\e\,dxd\xi \cr
&= \frac1\e \iint_{\om \times \R^2} \frac{x}{|x|}  \cdot \lt( (\xi - \e v^\e)f^\e + \nabla_\xi f^\e \rt) dxd\xi + \iint_{\om \times \R^2} \frac{x}{|x|} \cdot v^\e f^\e\,dxd\xi\cr
&=: I + II,
\end{align*}
where 
\[
|I| \leq  \frac1\e \iint_{\om \times \R^2} |(\xi - \e v^\e)f^\e + \nabla_\xi f^\e|\,dxd\xi \leq \lt( \frac1{\e^2}\iint_{\om \times \R^2} \frac1{f^\e}|  \nabla_\xi f^\e + (\xi - \e v^\e) f^\e|^2\,dxd\xi\rt)^\frac12
\]
and
\[
|II| \leq \lt(\int_\om \rho^\e|v^\e|^2\,dx\rt)^\frac12.
\]
Combining these bounds with Lemma \ref{lem:energy}, we obtain for any $\delta>0$,
\begin{align*}
 \iint_{\om \times \R^2} |x| f^\e(t)\,dxd\xi &  \leq \iint_{\om \times \R^2} |x| f^\e_0\,dxd\xi  + T^\frac12\lt( \frac1{\e^2}\int_0^T\iint_{\om \times \R^2} \frac1{f^\e}|  \nabla_\xi f^\e + (\xi - \e v^\e) f^\e|^2\,dxd\xi dt\rt)^\frac12\cr
&\quad + T^\frac12\lt(\int_0^T\int_\om \rho^\e|v^\e|^2\,dxdt\rt)^\frac12\cr
&\quad \leq \iint_{\om \times \R^2} |x| f^\e_0\,dxd\xi  + T^\frac12 \mathscr{F}[f^\e_0, v^\e_0]^\frac12 + \frac{T}{4\delta} + \delta\int_0^T\int_\om \rho^\e|v^\e|^2\,dxdt.
\end{align*}
To estimate the last term, we use Lemmas \ref{lem:TM} and \ref{lem:energy}:
\begin{align*}
\delta\int_0^T\int_\om \rho^\e|v^\e|^2\,dxdt &\leq C\delta\lt(1+\sup_{0 \leq t \leq T}\int_{\Omega} \rho^\e |\log \rho^\e|\,dx\rt) 
\|v^\e\|_{L^2(0,T;H^1(\Omega))}^2 \cr
&\leq C \delta  \mathscr{F}[f^\e_0, v^\e_0] \lt(1+\sup_{0\le t\le T}\int_{\Omega} \rho^\e |\log \rho^\e|\,dx\rt),
\end{align*}
where $C > 0$ is independent of $\e$. Thus,
\begin{align}\label{est:mom}
\begin{aligned}
\iint_{\om \times \R^2} |x| f^\e(t)\,dxd\xi &\leq \iint_{\om \times \R^2} |x| f^\e_0\,dxd\xi  + T^\frac12 \mathscr{F}[f^\e_0, v^\e_0]^\frac12 + \frac{T}{4\delta} \cr
&\quad + C \delta  \mathscr{F}[f^\e_0, v^\e_0] \lt(1+\sup_{0 \leq t \leq T}\int_{\Omega} \rho^\e |\log \rho^\e|\,dx\rt).
\end{aligned}
\end{align}

We now estimate the $L\log L$-entropy of $\rho^\e$. Choosing $(\rho,u)=(1,0)$ in the identity \eqref{eq_exp} gives
\[
  \intr f^\e \log \frac{f^\e}{M_{1, 0}}\,d\xi =  \intr f^\e \log \frac{f^\e}{M_{\rho^\e, u^\e}}\,d\xi +   \rho^\e \log \rho^\e + \frac{\rho^\e}{2}|u^\e|^2.
\]
Since
\[
  \intr f^\e \log \frac{f^\e}{M_{1, 0}}\,d\xi = \intr f^\e \log f^\e\,d\xi + \log (2\pi) \rho^\e + \frac12 \intr |\xi|^2 f^\e\,d\xi,
\]
we deduce
\[
 \intr f^\e \log \frac{f^\e}{M_{\rho^\e, u^\e}}\,d\xi + \rho^\e \log \rho^\e = \intr f^\e \log f^\e\,d\xi + \log (2\pi) \rho^\e + \frac12 \intr |\xi - u^\e|^2 f^\e\,d\xi.
\]
Using the identity
\[
\intr f^\e\,d\xi = \intr M_{\rho^\e, u^\e}\,d\xi,  
\]
we get
\[
\intr f^\e \log \frac{f^\e}{M_{\rho^\e, u^\e}}\,d\xi = \mathscr{H}[f^\e|M_{\rho^\e, u^\e}] \geq 0,
\]
and since
\[
\frac12 \intr |\xi - u^\e|^2 f^\e\,d\xi \leq \frac12 \intr |\xi|^2 f^\e\,d\xi,
\]
we have
\[
\int_\om \rho^\e \log \rho^\e \,dx \leq \mathscr{F}[f^\e,v^\e] + \log(2\pi) \leq \mathscr{F}[f^\e_0,v^\e_0] + \log(2\pi).
\]
To control the negative part, note that
\[
2\int_\Omega \rho^\e \log_-\rho^\e\,dx \leq \int_\Omega |x|\rho^\e \,dx + C \int_\Omega e^{-|x|}\,dx \leq \intrr |x| f^\e\,dxd\xi + C,
\]
where $\log_- \rho = \max\{0, -\log \rho\}$. Thus, we obtain
\[
\int_\Omega \rho^\e |\log \rho^\e|\,dx  \leq C + C\mathscr{F}[f^\e_0,v^\e_0] + C\intrr |x| f^\e\,dxd\xi.
\]
This combined with \eqref{est:mom} yields
\[
\sup_{0 \leq t \leq T}\iint_{\om \times \R^2} |x| f^\e(t)\,dxd\xi \leq C  \lt(1 + \|xf^\e_0\|_{L^1} + \mathscr{F}[f^\e_0, v^\e_0]\rt),
\]
for $\delta > 0$ small enough (such as $\delta = \frac{1}{M(1+\mathscr{F}[f^\e_0, v^\e_0])}$ with $M \gg 1$), where $C>0$ depends on $T$. Hence, we have
\[
\sup_{0 \leq t \leq T}\int_\Omega \rho^\e |\log \rho^\e|\,dx \leq C  \lt(1 + \|xf^\e_0\|_{L^1} + \mathscr{F}[f^\e_0, v^\e_0]\rt).
\]
This completes the proof.
\end{proof}
 
In the lemma below, we establish a uniform control on the second-order velocity moment of $f^\e$, which will later play an important role in the relative entropy estimates and compactness arguments.
\begin{lemma} \label{lem:kin} Let $T>0$, and let $(f^\e, v^\e)$ be the weak entropy solution to \eqref{main1} on the time interval $[0,T]$ in the sense of Definition \ref{def_weak}. Then,  we have
\[
\sup_{0 \leq t \leq T}\iint_{\om \times \R^2} |\xi|^2 f^\e(t)\,dxd\xi \leq C,
\]
where $C>0$ depends on $T$, $\||x|f^\e_0\|_{L^1}$, and $\mathscr{F}[f^\e_0, v^\e_0]$.
\end{lemma}
\begin{proof}  We use the standard decomposition 
\[
\log f^\e = \log_+ f^\e - \log_- f^\e, \quad \log_- f = \max\{0, -\log f\}. 
\]
The negative part can be compared with a Gaussian weight as follows:
\[
\iint_{\om \times \R^2} f^\e \log_{-} f^\e\,dx d\xi \leq \iint_{\om \times \R^2} f^\e \lt( \frac{|x|}{2} + \frac{|\xi|^2}{4} \rt)dx d\xi + C\iint_{\om \times \R^2} e^{-\frac{|x|}{4} - \frac{|\xi|^2}{8}}\,dxd\xi.
\]
This implies
\[
 \iint_{\om \times \R^2} f^\e \lt( \frac{|x|}{2} + \frac{|\xi|^2}{4} + \log_+ f^\e \rt)dx d\xi \leq \mathscr{F}[f^\e,v^\e] +   \iint_{\om \times \R^2} |x| f^\e\,dxd\xi + C_0,
\]
where  
\[
C_0 := C\iint_{\om \times \R^2} e^{-\frac{|x|}{4} - \frac{|\xi|^2}{8}}\,dxd\xi.
\]
We finally apply Lemmas \ref{lem:energy} and \ref{lem:mom} to conclude the desired bound estimate. 
\end{proof}

%
%
%
%
%
%

 
\section{Relative entropy inequality for the light-particle limit}\label{sec_stab}

%
%
%
%
%
%

In order to derive the stability structure of the light-particle limit, we first rewrite the macroscopic system \eqref{lim_light} in a form that exposes its relaxation character.

Introducing the auxiliary velocity
\bq\label{bdd_u}
{\rm u}_\e = \e (v - \nabla \log \rho),
\eq
the limit equations \eqref{lim_light} may be equivalently written as 
\begin{align}\label{int_light}
\begin{aligned}
&\pa_t \rho + \frac1\e \nabla \cdot (\rho {\rm u}_\e) =0, \cr
&\rho \pa_t {\rm u}_\e + \frac1\e \rho {\rm u}_\e \cdot \nabla {\rm u}_\e + \frac1\e \nabla  \rho = \frac1\e\rho v - \frac1{\e^2} \rho {\rm u}_\e + \rho {\rm e}_\e,\cr
&\pa_t v + v \cdot \nabla v + \nabla_x P -  \Delta v = 0, \cr
&\nabla \cdot v = 0,
\end{aligned}
\end{align}
where 
\[
\rho {\rm e}_\e = \rho \pa_t {\rm u}_\e + \frac1\e \rho {\rm u}_\e \cdot \nabla {\rm u}_\e.
\]
We notice that ${\rm u}_\e$ is an $O(\e)$ quantity and thus,
\[
\|{\rm e}_\e\|_{L^\infty} \leq \e \lt( \|\pa_t  (v - \nabla \log \rho)\|_{L^\infty} + \|(v - \nabla \log \rho) \cdot \nabla (v - \nabla \log \rho)\|_{L^\infty} \rt).
\]

\begin{proposition}\label{prop_key_light} Let $T>0$, and let $(f^\e, v^\e)$ be the weak entropy solution to \eqref{main1} on the time interval $[0,T]$ in the sense of Definition \ref{def_weak}. Let $(\rho,v)$ be the sufficiently regular solution to the system \eqref{lim_light}. Then we have
\begin{align*}
&\mathscr{H}[f^\e | M_{\rho, {\rm u}_\e}](t) + \frac12\|(v^\e - v)(t)\|_{L^2}^2 + \frac1{\e^2  } \int_0^t \iint_{\om \times \R^2} \frac1{f^\e} |  \nabla_\xi f^\e + (\xi - u^\e)f^\e|^2\,dxd\xi ds   \cr
&\quad + \frac{1}{\e^2}\int_0^t \int_\om \rho^\e | ({\rm u}_\e - u^\e) - \e (v - v^\e)|^2\,dx ds + \int_0^t \int_\om |\nabla (v^\e - v)|^2\,dxds  \cr
&\quad \leq  \mathscr{H}[f^\e | M_{\rho, {\rm u}_\e}](0) + \frac12\|(v^\e - v)(0)\|_{L^2}^2 \cr
&\quad \quad - \frac1{\e  }\int_0^t \iint_{\om \times \R^2} \lt(f^\e ({\rm u}_\e-\xi) \otimes ({\rm u}_\e-\xi)  -    f^\e \mathbb{I}_2\rt): \nabla {\rm u}_\e\,dxd\xi ds   +   \int_0^t \int_\om (v^\e - v)\otimes (v^\e - v) : \nabla v\,dxds\cr
&\quad \quad   +  \int_0^t \int_\om (\rho^\e - \rho) (v^\e - v) \cdot \nabla \log \rho\,dxds + \int_0^t \int_\om \rho^\e ({\rm u}_\e - u^\e) \cdot {\rm e}_\e\,dxds,
\end{align*}
for a.e. $t \in [0,T]$.
\end{proposition}
\begin{proof}
We provide only a formal derivation of the relative entropy identity.  At the formal level, the computation below yields an exact balance law. For entropy weak solutions, however, one must apply standard truncation and regularization procedures (see, e.g., \cite{CKL11, GJV04, Ham98}).  In this process, several integration-by-parts steps are justified only for regularized approximations, and the passage to the limit typically uses lower semicontinuity of the dissipation terms.  Consequently, the rigorous relative entropy balance for weak solutions holds in the form of an inequality rather than an identity.

We begin with
\begin{align*}
\frac{d}{dt}\iint_{\om \times \R^2} f^\e \log \lt( \frac{f^\e}{ M_{\rho, {\rm u}_\e}} \rt)dxd\xi &= \iint_{\om \times \R^2} \pa_t f^\e \lt(  \log f^\e - \log M_{\rho, {\rm u}_\e}\rt)dxd\xi \cr
&\quad + \iint_{\om \times \R^2} f^\e \lt( \frac{\pa_t f^\e}{f^\e}  - \frac{\pa_t M_{\rho, {\rm u}_\e}}{M_{\rho, {\rm u}_\e}}\rt) dxd\xi\cr
&=: I + II,
\end{align*}
where
\begin{align*}
I &= \frac1\e \iint_{\om \times \R^2} \lt(  \log f^\e - \log M_{\rho, {\rm u}_\e}\rt)   \lt( - \xi \cdot \nabla  f^\e    + \frac1\e \nabla_\xi \cdot ( \nabla_\xi f^\e + (\xi - \e v^\e) f^\e)\rt) dxd\xi,\cr
II &= - \iint_{\om \times \R^2} f^\e \lt( \frac{\pa_t \rho}{\rho}  - ({\rm u}_\e-\xi) \cdot \pa_t {\rm u}_\e\rt)  dxd\xi.
\end{align*}
For $I$, we use the kinetic equation for $f^\e$ to get
\begin{align*}
I = I_1 + I_2,
\end{align*}
with
\begin{align*}
I_1 &:= \frac1\e \iint_{\om \times \R^2} \xi f^\e \cdot \lt( \frac{\nabla  f^\e}{f^\e}  - \frac{\nabla  M_{\rho, {\rm u}_\e}}{M_{\rho, {\rm u}_\e}}\rt) dxd\xi \cr
&= - \frac1\e \iint_{\om \times \R^2} \xi f^\e \cdot \lt( \nabla  \log \rho - (\nabla  {\rm u}_\e) ({\rm u}_\e-\xi) \rt) dxd\xi\cr
&= - \frac1\e \int_\om \rho^\e u^\e \cdot \nabla  \log \rho\,dx + \frac1{\e }\iint_{\om \times \R^2} \xi f^\e \cdot (\nabla  {\rm u}_\e)({\rm u}_\e-\xi)\,dxd\xi 
\end{align*}
and
\begin{align*}
I_2 &:= - \frac1{\e^2} \iint_{\om \times \R^2} \lt(  \nabla_\xi f^\e + (\xi - \e v^\e) f^\e\rt) \cdot \lt( \frac{\nabla_\xi f^\e}{f^\e} - \frac{\nabla_\xi M_{\rho, {\rm u}_\e}}{M_{\rho, {\rm u}_\e}}\rt) dxd\xi\cr
&= - \frac1{\e^2} \iint_{\om \times \R^2} \lt(  \nabla_\xi f^\e + (\xi - \e v^\e) f^\e\rt) \cdot \lt( \frac{\nabla_\xi f^\e}{f^\e} -  ({\rm u}_\e-\xi) \rt) dxd\xi.
\end{align*}
On the other hand, we find from Remark \ref{rmk_diss} that
\begin{align*}
&\iint_{\om \times \R^2} \lt(  \nabla_\xi f^\e + (\xi - \e v^\e) f^\e\rt) \cdot \lt( \frac{\nabla_\xi f^\e}{f^\e} +  (\xi - {\rm u}_\e) \rt) dxd\xi\cr
&\quad =  \iint_{\om \times \R^2} \lt(  \nabla_\xi f^\e + (\xi - \e v^\e) f^\e\rt) \cdot \lt( \frac{\nabla_\xi f^\e}{f^\e} +   (\xi - \e v^\e) \rt)  dxd\xi \cr
& \quad + \iint_{\om \times \R^2} \lt(  \nabla_\xi f^\e + (\xi - \e v^\e) f^\e\rt) \cdot \lt( \e v^\e - {\rm u}_\e\rt) dxd\xi\cr
&\quad =  \iint_{\om \times \R^2} \frac1{f^\e} | \nabla_\xi f^\e + (\xi - \e v^\e) f^\e|^2\,dxd\xi - \int_\om \rho^\e (u^\e - \e v^\e)\cdot ({\rm u}_\e - \e v^\e)\,dx\cr
&\quad =     \iint_{\om \times \R^2} \frac1{f^\e}|  \nabla_\xi f^\e + (\xi - u^\e) f^\e|^2\,dxd\xi +   \int_\om \rho^\e (u^\e - \e v^\e) \cdot (u^\e - {\rm u}_\e)\,dx.
\end{align*}
This implies
\begin{align*}
I &= - \frac1\e\int_\om \rho^\e u^\e \cdot \nabla  \log \rho\,dx +\frac1{\e } \iint_{\om \times \R^2} f^\e \xi  \cdot (\nabla  {\rm u}_\e)({\rm u}_\e-\xi)\,dxd\xi \cr
&\quad  - \frac1{\e^2  } \iint_{\om \times \R^2} \frac1{f^\e} |  \nabla_\xi f^\e + (\xi - u^\e)f^\e|^2\,dxd\xi + \frac1{\e^2  } \int_\om \rho^\e (u^\e - \e v^\e) \cdot ({\rm u}_\e - u^\e)\,dx.
\end{align*}
For $II$, we get
\[
II = II_1 + II_2,
\]
where we use the reformulated system \eqref{int_light} to obtain
\[
II_1 :=  \frac1\e \int_\om \frac{\rho^\e}{\rho} \nabla \cdot (\rho {\rm u}_\e) \,dx=  \frac1\e\int_\om \rho^\e {\rm u}_\e \cdot \nabla \log \rho\,dx + \frac1\e \int_\om \rho^\e \mathbb{I}_2 : \nabla  {\rm u}_\e\,dx  
\]
and
\begin{align*}
II_2 &:=   \iint_{\om \times \R^2} f^\e ({\rm u}_\e-\xi) \cdot \lt( -  \frac1\e({\rm u}_\e \cdot \nabla) {\rm u}_\e -  \frac1\e \nabla \log \rho +  \frac1\e v -  \frac1{\e^2} {\rm u}_\e + {\rm e}_\e \rt)  dxd\xi\cr
&= - \frac1{\e } \iint_{\om \times \R^2} f^\e {\rm u}_\e \cdot (\nabla {\rm u}_\e)( {\rm u}_\e-\xi)\,dxd\xi - \frac1\e \int_\om \rho^\e ({\rm u}_\e - u^\e) \cdot \nabla \log \rho\,dx \cr
&\quad  + \frac1{\e } \iint_{\om \times \R^2} f^\e({\rm u}_\e-\xi) \cdot v\,dxd\xi - \frac1{\e^2  } \int_\om \rho^\e ({\rm u}_\e - u^\e) \cdot {\rm u}_\e \,dx +  \int_\om \rho^\e ({\rm u}_\e - u^\e) \cdot {\rm e}_\e\,dx.
\end{align*}
Combining the estimates of $II_1$ and $II_2$ gives
\begin{align*}
II &= \frac1\e \int_\om \rho^\e u^\e \cdot \nabla  \log \rho\,dx - \frac1{\e} \iint_{\om \times \R^2} f^\e {\rm u}_\e \cdot (\nabla  {\rm u}_\e) ({\rm u}_\e - \xi)\,dxd\xi + \frac1\e \int_\om \rho^\e \mathbb{I}_2 : \nabla  {\rm u}_\e\,dx\cr
&\quad   + \frac1{\e^2 }\int_\om \rho^\e({\rm u}_\e-u^\e) \cdot (\e v - {\rm u}_\e)\,dx +  \int_\om \rho^\e ({\rm u}_\e - u^\e) \cdot {\rm e}_\e\,dx.
\end{align*}
Hence, we have
\begin{align*}
&\frac{d}{dt}\iint_{\om \times \R^2} f^\e \log \lt( \frac{f^\e}{ M_{\rho, {\rm u}_\e}} \rt) dxd\xi + \frac1{\e^2  } \iint_{\om \times \R^2} \frac1{f^\e} |  \nabla_\xi f^\e + (\xi - u^\e)f^\e|^2\,dxd\xi + \frac1{\e^2  } \int_\om \rho^\e |{\rm u}_\e - u^\e|^2\,dx  \cr
&\quad =  - \frac1{\e } \iint_{\om \times \R^2} f^\e ({\rm u}_\e-\xi) \otimes ({\rm u}_\e-\xi) : \nabla {\rm u}_\e \,dxd\xi+ \frac1\e \int_\om \rho^\e \mathbb{I}_2 : \nabla {\rm u}_\e\,dx\cr
& \quad   + \frac1{\e  } \int_\om \rho^\e ({\rm u}_\e - u^\e) \cdot (v - v^\e)\,dx   + \int_\om \rho^\e ({\rm u}_\e - u^\e) \cdot {\rm e}_\e\,dx.
\end{align*}

For the fluid part, we deduce
\begin{align*}
&\frac12\frac{d}{dt}\int_\om |v^\e - v|^2\,dx +  \int_\om |\nabla (v^\e - v)|^2\,dx \cr
&\quad = \int_\om (v^\e - v)\otimes (v^\e - v) : \nabla v\,dx + \frac1\e \int_\om \rho^\e (v^\e - v) \cdot  (u^\e - \e v^\e )\,dx.
\end{align*}

Combining both identities implies
\begin{align*}
&\frac{d}{dt}\lt( \iint_{\om \times \R^2} f^\e \log \lt( \frac{f^\e}{ M_{\rho, {\rm u}_\e}} \rt) dxd\xi  + \frac1{2 }\int_\om |v^\e - v|^2\,dx \rt) \cr
&\quad + \frac1{\e^2  } \iint_{\om \times \R^2} \frac1{f^\e} |  \nabla_\xi f^\e + (\xi - u^\e)f^\e|^2\,dxd\xi + \frac1{\e^2  } \int_\om \rho^\e |{\rm u}_\e - u^\e|^2\,dx +    \int_\om |\nabla (v^\e - v)|^2\,dx  \cr
&\quad =  - \frac1{\e } \iint_{\om \times \R^2} f^\e ({\rm u}_\e-\xi) \otimes ({\rm u}_\e-\xi) : \nabla {\rm u}_\e\,dxd\xi + \frac1\e \int_\om \rho^\e \mathbb{I}_2 : \nabla  {\rm u}_\e \,dx+  \int_\om (v^\e - v)\otimes (v^\e - v) : \nabla v\,dx\cr
&\quad \quad   + \frac1{\e  } \int_\om \rho^\e ({\rm u}_\e - u^\e) \cdot (v - v^\e)\,dx  + \frac1{\e } \int_\om \rho^\e (v^\e - v) \cdot  (u^\e - \e v^\e )\,dx +  \int_\om \rho^\e ({\rm u}_\e - u^\e) \cdot {\rm e}_\e\,dx.
\end{align*}
Here we further estimate
\begin{align*}
&\frac1{\e  } \int_\om \rho^\e ({\rm u}_\e - u^\e) \cdot (v - v^\e)\,dx  + \frac1{\e } \int_\om \rho^\e (v^\e - v) \cdot  (u^\e - \e v^\e)\,dx \cr
&\quad = \frac2{\e  } \int_\om \rho^\e ({\rm u}_\e - u^\e) \cdot (v - v^\e)\,dx + \frac1{\e } \int_\om \rho^\e (v^\e - v) \cdot  ({\rm u}_\e - \e v^\e)\,dx\cr
&\quad = - \frac{1}{\e^2  }\int_\om \rho^\e | ({\rm u}_\e - u^\e) - \e (v - v^\e)|^2\,dx + \frac{1}{\e^2  } \int_\om \rho^\e |{\rm u}_\e - u^\e|^2\,dx\cr
& \quad  +   \int_\om \rho^\e | v - v^\e|^2\,dx + \frac1{\e } \int_\om \rho^\e (v^\e - v) \cdot  ({\rm u}_\e - \e v^\e)\,dx\cr
&\quad = - \frac{1}{\e^2  }\int_\om \rho^\e | ({\rm u}_\e - u^\e) - \e (v - v^\e)|^2\,dx + \frac{1}{\e^2  } \int_\om \rho^\e |{\rm u}_\e - u^\e|^2\,dx  + \frac1{\e } \int_\om \rho^\e (v^\e - v) \cdot  ({\rm u}_\e - \e v)\,dx
\end{align*}
and
\[
\frac1{\e } \int_\om \rho^\e (v^\e - v) \cdot  ({\rm u}_\e - \e v)\,dx = \frac1{\e } \int_\om (\rho^\e - \rho) (v^\e - v) \cdot  ({\rm u}_\e - \e v)\,dx = \int_\om (\rho^\e - \rho) (v^\e - v) \cdot \nabla \log \rho\,dx,
\]
where we used
\[
\int_\om \rho (v^\e - v) \cdot  ({\rm u}_\e - \e v)\,dx = -\e \int_\om (v^\e - v) \cdot \nabla \rho\,dx = 0
\]
due to 
\[
\rho {\rm u}_\e = \e (\rho v - \nabla \rho).
\]
This yields
\begin{align*}
&\frac{d}{dt}\lt( \iint_{\om \times \R^2} f^\e \log \lt( \frac{f^\e}{ M_{\rho, {\rm u}_\e}} \rt)  dxd\xi + \frac1{2 }\int_\om |v^\e - v|^2\,dx \rt) \cr
&\quad + \frac1{\e^2  } \iint_{\om \times \R^2} \frac1{f^\e} |  \nabla_\xi f^\e + (\xi - u^\e)f^\e|^2\,dxd\xi + \frac{1}{\e^2  }\int_\om \rho^\e | ({\rm u}_\e - u^\e) - \e (v - v^\e)|^2\,dx +   \int_\om |\nabla (v^\e - v)|^2\,dx  \cr
&\quad =  - \frac1{\e } \iint_{\om \times \R^2} \lt(f^\e ({\rm u}_\e-\xi) \otimes ({\rm u}_\e-\xi)  -   f^\e \mathbb{I}_2\rt): \nabla {\rm u}_\e\,dxd\xi   +   \int_\om (v^\e - v)\otimes (v^\e - v) : \nabla v\,dx\cr
&\quad \quad   +  \int_\om (\rho^\e - \rho) (v^\e - v) \cdot \nabla \log \rho\,dx +   \int_\om \rho^\e ({\rm u}_\e - u^\e) \cdot {\rm e}_\e\,dx.
\end{align*}
Integrating the above with respect to time concludes the desired result.
\end{proof}

%
%
%
%
%
%

\section{Proof of Theorem \ref{thm_main}}\label{sec_proof}

We begin the proof by comparing the two relative entropies $\mathscr{H}[f^\e|M_{\rho,0}]$ and $\mathscr{H}[f^\e|M_{\rho,{\rm u}_\e}]$, showing that they differ only by $O(\e^2)$. This equivalence will allow us to work with the more convenient Maxwellian $M_{\rho,{\rm u}_\e}$ in the stability estimate.

\begin{lemma}\label{lem:equi} 
There exists a constant $C >0$, independent of $\e>0$, such that  
\[
\frac12 \mathscr{H}[f^\e | M_{\rho, 0}] - C\e^2 \leq \mathscr{H}[f^\e | M_{\rho, {\rm u}_\e}] \leq 2 \mathscr{H}[f^\e | M_{\rho, 0}] + C\e^2.
\]
\end{lemma}
\begin{proof} 
 Since
\begin{align*}
\intrr  f^\e \log \frac{M_{\rho, {\rm u}_\e}}{M_{\rho, 0}}\,dxd\xi  
 = \frac1{2} \int_\Omega \rho^\e (2 {\rm u}_\e \cdot u^\e - |{\rm u}_\e|^2)\,dx  
 = -\int_\Omega \rho^\e {\rm u}_\e \cdot (u^\e - {\rm u}_\e)\,dx + \frac1{2}\int_\Omega \rho^\e |{\rm u}_\e|^2\,dx,
\end{align*}
we obtain
\begin{align*}
\lt|\mathscr{H}[f^\e | M_{\rho, 0}] - \mathscr{H}[f^\e | M_{\rho, {\rm u}_\e}]\rt| 
   \leq 2\int_\Omega \rho^\e|{\rm u}_\e|^2\,dx + \frac1{4}\int_\Omega \rho^\e |u^\e - {\rm u}_\e|^2\,dx
  \leq C\e^2 + \frac12\mathscr{H}[f^\e | M_{\rho, {\rm u}_\e}]
\end{align*}
due to \eqref{bdd_u}. Hence the two entropies are equivalent up to $O(\e^2)$, namely
\[
\mathscr{H}[f^\e | M_{\rho, {\rm u}_\e}] \leq 2 \mathscr{H}[f^\e | M_{\rho, 0}] + C\e^2
\]
and by symmetry,
\[
\mathscr{H}[f^\e | M_{\rho, 0}] \leq 2 \mathscr{H}[f^\e | M_{\rho, {\rm u}_\e}] + C\e^2.
\]
This completes the proof.
\end{proof}

%
%
%
%
%
%
 \subsection{Closing the relative entropy inequality} We now turn to the main step: estimating the error terms produced by Proposition \ref{prop_key_light} and closing the relative entropy inequality by a Gr\"onwall argument.

From Proposition \ref{prop_key_light}, for a.e. $t \in [0,T]$, we have
\begin{align*}
&\mathscr{H}[f^\e | M_{\rho, {\rm u}_\e}](t) + \frac12\|(v^\e - v)(t)\|_{L^2}^2 + \frac1{\e^2  } \int_0^t \iint_{\om \times \R^2} \frac1{f^\e} |  \nabla_\xi f^\e + (\xi - u^\e)f^\e|^2\,dxd\xi ds   \cr
&\quad + \frac{1}{\e^2}\int_0^t \int_\om \rho^\e | ({\rm u}_\e - u^\e) - \e (v - v^\e)|^2\,dx ds + \int_0^t \int_\om |\nabla (v^\e - v)|^2\,dxds  \cr
&\quad \leq   \mathscr{H}[f^\e | M_{\rho, {\rm u}_\e}](0) + \frac12\|(v^\e - v)(0)\|_{L^2}^2  + I + II + III + IV,
\end{align*}
where
\begin{align*}
I &:= - \frac1{\e  }\int_0^t \iint_{\om \times \R^2} \lt(f^\e ({\rm u}_\e-\xi) \otimes ({\rm u}_\e-\xi)  -    f^\e \mathbb{I}_2\rt): \nabla {\rm u}_\e\,dxd\xi ds,\cr
II &:= \int_0^t \int_\om (v^\e - v)\otimes (v^\e - v) : \nabla v\,dxds,\cr
III &:= \int_0^t \int_\om (\rho^\e - \rho) (v^\e - v) \cdot \nabla \log \rho\,dxds,\cr
IV &:= \int_0^t \int_\om \rho^\e ({\rm u}_\e - u^\e) \cdot {\rm e}_\e\,dxds.
\end{align*}

We estimate each term separately. 

For $I$, we use the identity
\[
\intr f^\e ({\rm u}_\e-\xi) \otimes ({\rm u}_\e-\xi)\,d\xi =  \rho^\e ({\rm u}_\e - u^\e) \otimes ({\rm u}_\e - u^\e) - \intr f^\e (u^\e - \xi) \otimes \xi\,d\xi
\]
to decompose 
\begin{align*}
&- \iint_{\om \times \R^2}  \lt( f^\e ({\rm u}_\e-\xi) \otimes ({\rm u}_\e-\xi)  -   f^\e \mathbb{I}_2 \rt): \nabla {\rm u}_\e\,dxd\xi  \cr
&\quad = - \int_\om \rho^\e ({\rm u}_\e - u^\e) \otimes ({\rm u}_\e - u^\e) : \nabla {\rm u}_\e\,dx   + \int_\om \lt(  \int_\om \lt( f^\e (u^\e - \xi) -  \nabla_\xi f^\e \rt) \otimes \xi\,d\xi \rt) : \nabla  {\rm u}_\e\,dx.
\end{align*}
Note that the right-hand side of the above can be bounded as
\begin{align*}
\lt|\int_\om \rho^\e ({\rm u}_\e - u^\e) \otimes ({\rm u}_\e - u^\e) : \nabla {\rm u}_\e\,dx\rt| &\leq \|\nabla {\rm u}_\e\|_{L^\infty} \int_\om \rho^\e |u^\e - {\rm u}_\e|^2\,dx \cr
&= \e \|\nabla (v - \nabla \log \rho)\|_{L^\infty} \int_\om \rho^\e |u^\e - {\rm u}_\e|^2\,dx
\end{align*}
and
\begin{align*}
&\lt| \int_\om \lt(  \int_\om \lt( f^\e (u^\e - \xi) -  \nabla_\xi f^\e \rt) \otimes \xi\,d\xi \rt) : \nabla  {\rm u}_\e\,dx \rt| \cr
&\quad \leq \|\nabla {\rm u}_\e\|_{L^\infty}\lt(\iint_{\om \times \R^2} |\xi|^2 f^\e\,dxd\xi  \rt)^\frac12 \lt( \iint_{\om \times \R^2} \frac1{f^\e} |   \nabla_\xi f^\e - (u^\e - \xi)f^\e|^2\,dxd\xi \rt)^\frac12\cr
&\quad = \e \|\nabla (v - \nabla \log \rho)\|_{L^\infty} \|f^\e\|_{L^1_2}^\frac12  \lt( \iint_{\om \times \R^2} \frac1{f^\e} |   \nabla_\xi f^\e - (u^\e - \xi)f^\e|^2\,dxd\xi \rt)^\frac12.
\end{align*}
Thus, by Lemma \ref{lem:kin}, we deduce
\[
|I| \leq C  \int_0^t \int_\om \rho^\e |u^\e - {\rm u}_\e|^2\,dx ds +  C\e^2   + \frac1{4\e^2  } \int_0^t \iint_{\om \times \R^2} \frac1{f^\e} |  \nabla_\xi f^\e + (\xi - u^\e)f^\e|^2\,dxd\xi ds ,
\]
where $C>0$ depends on $\|\nabla (v - \nabla \log \rho)\|_{L^\infty}$.

For $II$, we readily find
\[
|II| \leq  \|\nabla v\|_{L^\infty} \int_0^t  \int_\om | v^\e - v|^2\,dxds,
\]
and by Lemmas \ref{lem:TM} and \ref{lem:mom} we bound $III$ as
\begin{align*}
|III| &\leq  \|\nabla \log \rho \|_{L^\infty}\lt( 2 \int_0^t \int_\om \rho^\e \log \lt(\frac{\rho^\e}\rho \rt) dx ds\rt)^{\frac12} \lt(\int_0^t \int_\om (\rho + \rho^\e)|(v^\e - v)|^2\,dx ds \rt)^{\frac12}\cr
&\leq  C\int_0^t \int_\om \rho^\e \log \lt(\frac{\rho^\e}\rho \rt) dxds + C\int_0^t\int_\om |v^\e - v|^2\,dxds  + \frac14\int_0^t\int_\om |\nabla (v^\e - v)|^2\,dxds,
\end{align*}
where $C > 0$ depends on $\|\nabla \log \rho \|_{L^\infty}$ and $\|\rho\|_{L^\infty}$. 

Finally, we estimate
\[
|IV| \leq \frac1{2\e}\int_0^t\int_\om \rho^\e |u^\e - {\rm u}_\e|^2\,dxds +  \e\frac{\| {\rm e}_\e\|_{L^\infty}^2}2\|\rho^\e\|_{L^1} \leq  \frac1{2\e}\int_0^t\int_\om \rho^\e |u^\e - {\rm u}_\e|^2\,dxds + C\e^2,
\]
where $C>0$ depends on $\|(v - \nabla \log \rho)\|_{W^{1,\infty}_{t,x}}$. Here, by Lemma \ref{lem:TM}, we further estimate
\begin{align*}
\frac1{2\e}\int_0^t\int_\om \rho^\e |u^\e - {\rm u}_\e|^2\,dxds &\leq \frac1\e\int_0^t \int_\om \rho^\e | ({\rm u}_\e - u^\e) - \e (v - v^\e)|^2\,dxds + \e \int_0^t\int_\om \rho^\e |v- v^\e|^2\,dxds\cr
&\leq \frac1\e \int_0^t\int_\om \rho^\e | ({\rm u}_\e - u^\e) - \e (v - v^\e)|^2\,dxds +  C\int_0^t\int_\om |v^\e - v|^2\,dxds \cr
&\quad + \frac14\int_0^t\int_\om |\nabla (v^\e - v)|^2\,dxds
\end{align*}
for sufficiently small $\e>0$. 

Collecting all bounds gives
\begin{align*}
&\mathscr{H}[f^\e | M_{\rho, {\rm u}_\e}](t) + \|(v^\e - v)(t)\|_{L^2}^2 + \frac1{2\e^2  } \int_0^t \iint_{\om \times \R^2} \frac1{f^\e} |  \nabla_\xi f^\e + (\xi - u^\e)f^\e|^2\,dxd\xi ds   \cr
&\quad + \frac{1}{2\e^2}\int_0^t \int_\om \rho^\e | ({\rm u}_\e - u^\e) - \e (v - v^\e)|^2\,dx ds + \frac12 \int_0^t \int_\om |\nabla (v^\e - v)|^2\,dxds  \cr
&\quad  \leq C\e^2 +  C\mathscr{H}[f^\e | M_{\rho, {\rm u}_\e}](0) + C\|(v^\e - v)(0)\|_{L^2}^2 + C \int_0^t \lt(\mathscr{H}[f^\e | M_{\rho, {\rm u}_\e}](s) + \|(v^\e - v)(s)\|_{L^2}^2\rt) ds.
\end{align*}
Applying Gr\"onwall's lemma to the above yields
\begin{align}\label{rel}
\begin{aligned}
&\mathscr{H}[f^\e | M_{\rho, {\rm u}_\e}](t)  + \|(v^\e - v)(t)\|_{L^2}^2 + \frac1{2\e^2  } \int_0^t \iint_{\om \times \R^2} \frac1{f^\e} |  \nabla_\xi f^\e + (\xi - u^\e)f^\e|^2\,dxd\xi ds   \cr
&\quad + \frac{1}{2\e^2}\int_0^t \int_\om \rho^\e | ({\rm u}_\e - u^\e) - \e (v - v^\e)|^2\,dx ds + \frac12 \int_0^t \int_\om |\nabla (v^\e - v)|^2\,dxds  \cr
&\quad  \leq C\e^2 +  C\mathscr{H}[f^\e | M_{\rho, {\rm u}_\e}](0)  + C\|v^\e_0 - v_0\|_{L^2}^2 .
\end{aligned}
\end{align}

%
%
%
%
%
%

\subsection{Stability via relative entropy} Under the well-preparedness assumption on the initial data, the 
relative entropy inequality \eqref{rel} yields 
\[
\sup_{0 \leq t \leq T} \lt( \mathscr{H}[f^\e | M_{\rho, {\rm u}_\e}](t)  + \|(v^\e - v)(t)\|_{L^2}^2\rt) + \|\nabla (v^\e - v)\|_{L^2([0,T] \times \Omega)}^2 \to 0  \quad \text{as } \e \to 0.
\]
Thus, as mentioned in Remark \ref{rmk_expan}, by Csisz\'ar--Kullback--Pinsker inequality, we obtain the convergences
\begin{align*}
f^\e &\to M_{\rho, 0} \quad \mbox{in } L^\infty(0,T; L^1(\Omega \times \R^2)), \cr
 \intr f^\e\,d\xi &\to \rho \quad \mbox{in }  L^\infty(0,T; L^1(\Omega)), \cr
  v^\e &\to v \quad \mbox{in } L^\infty(0,T; L^2(\Omega)) \cap L^2(0,T;H^1(\om)).
\end{align*}
For the convergence of scaled momentum density
\[
\frac1\e \intr \xi f^\e\,d\xi = \frac1\e \rho^\e u^\e,
\]
we begin with
\begin{align*}
 \int_\Omega | \rho^\e u^\e - \rho {\rm u}_\e | \,dx &\leq   \int_\Omega \rho^\e | u^\e - {\rm u}_\e|\, dx +   \int_\Omega |\rho^\e - \rho| |{\rm u}_\e|\, dx \cr
 &\leq \lt(  \int_\Omega \rho^\e | u^\e - {\rm u}_\e|^2\, dx\rt)^\frac12 + \e\| v- \nabla \log \rho \|_{L^\infty} \|\rho^\e - \rho\|_{L^1}.
 \end{align*}
 Thus, the key quantity to control is the modulated kinetic energy, 
\[
\frac1{2\e^2}\int_0^t\int_\om \rho^\e |u^\e - {\rm u}_\e|^2\,dxds.
\] 
Using the dissipation terms in the relative entropy inequality \eqref{rel} and Lemmas \ref{lem:TM} and \ref{lem:mom}, we deduce
\begin{align}\label{kin_ref}
\begin{aligned}
\frac1{2\e^2}\int_0^t\int_\om \rho^\e |u^\e - {\rm u}_\e|^2\,dxds &\leq \frac1{\e^2}\int_0^t \int_\om \rho^\e | ({\rm u}_\e - u^\e) - \e (v - v^\e)|^2\,dxds + \int_0^t\int_\om \rho^\e |v- v^\e|^2\,dxds\cr
&\leq \frac1{\e^2} \int_0^t\int_\om \rho^\e | ({\rm u}_\e - u^\e) - \e (v - v^\e)|^2\,dxds +  C\int_0^t\int_\om |v^\e - v|^2\,dxds \cr
&\quad + C\int_0^t\int_\om |\nabla (v^\e - v)|^2\,dxds\cr
&\leq C\e^2 +  C\mathscr{H}[f^\e | M_{\rho, 0}](0)  + C\|v^\e_0 - v_0\|_{L^2}^2. 
\end{aligned}
\end{align}
Inserting \eqref{kin_ref} into the previous bound for the momentum density gives
 \[
\frac1\e\int_0^T\int_\Omega | \rho^\e u^\e - \rho {\rm u}_\e | \,dxdt \leq C \lt(  \frac1{\e^2}\int_0^T\int_\Omega \rho^\e | u^\e - {\rm u}_\e|^2\, dxdt\rt)^\frac12 + C\sup_{0 \leq t \leq T} \|(\rho^\e - \rho)(t)\|_{L^1},
 \]
and thus,
\[
\frac1\e\int_0^T\int_\Omega | \rho^\e u^\e - \rho {\rm u}_\e | \,dxdt \leq C\e + C\lt( \mathscr{H}[f^\e | M_{\rho, 0}](0) +\|v^\e_0 - v_0\|_{L^2}^2\rt)^\frac12.
\]

Hence, we conclude the desired stability and strong convergence estimates.

%
%
%
%
%
%

\subsection{Stability in negative Sobolev spaces}\label{sec_weak-conv}

In this subsection, we derive convergence estimates for $v^\e-v$ and $\rho^\e-\rho$ in negative Sobolev spaces. Although the two estimates rely on the same heat-kernel representation formula, the nonlinear source terms are controlled by different arguments. For this reason, we first record a common parabolic representation formula and then apply it separately to the velocity and density equations.

For simplicity of presentation, we provide the argument only in the whole-space case $\Omega=\R^2$. The periodic case $\Omega=\T^2$ can be treated in the same way, replacing the heat kernel on $\R^2$ by the corresponding periodic heat semigroup; the resulting estimates are entirely analogous.
 
We denote by
\[
\mh_t(x):=\frac{1}{4\pi t}e^{-\frac{|x|^2}{4t}}, \quad t>0,\ x\in\R^2,
\]
the heat kernel on $\R^2$.  

\begin{lemma}\label{lem:para-rep}
Let $s>1$ and $T>0$. Assume that
\[
w_0\in H^{-s-1}(\R^2), \quad F,H\in L^1(0,T;H^{-s}(\R^2;\R^2)), \quad G\in L^1(0,T;H^{-s}(\R^2;\R^{2\times 2})),
\]
and let
\[
w\in L^\infty(0,T;H^{-s-1}(\R^2))
\]
be a distributional solution to
\bq\label{eq:para-gene}
\pa_t w-\Delta w =\nabla\cdot F+\nabla\otimes\nabla:G+\pa_t H \quad\text{in }(0,T)\times\R^2,
\eq
with initial data
\[
w|_{t=0}=w_0
\]
in $H^{-s-1}(\R^2)$. Then, for every $\phi\in H^{s+1}(\R^2)$ and almost every $t\in[0,T]$, we have
\bq\label{eq:para-rep}
\begin{aligned}
\lal w(t),\phi\ral
&=\lal w_0,\mh_t\star_x \phi\ral
- \int_0^t \lal F(\tau),\nabla(\mh_{t-\tau}\star_x \phi)\ral\,d\tau   +\int_0^t \lal G(\tau),\nabla\otimes\nabla(\mh_{t-\tau}\star_x \phi)\ral\,d\tau \\
&\quad +\lal H(t),\phi\ral
-\lal H(0),\mh_t\star_x \phi\ral  + \int_0^t \lal H(\tau),\Delta(\mh_{t-\tau}\star_x \phi)\ral\,d\tau,
\end{aligned}
\eq
where the brackets denote the natural duality pairings and $\star_x$ denotes the convolution with respect to $x$.
\end{lemma}

\begin{proof}
Let $t\in(0,T]$ and let $\phi\in C_c^\infty(\R^2)$. Define
\[
\Phi(\tau,x):=(\mh_{t-\tau}\star_x\phi)(x),  \quad 0\le \tau<t.
\]
Then $\Phi$ satisfies
\[
-\pa_\tau \Phi-\Delta\Phi=0 \quad\text{in }(0,t)\times\R^2, \quad \Phi(t,\cdot)=\phi.
\]
Testing \eqref{eq:para-gene} against $\Phi$ on $(0,t)\times\R^2$, we obtain
\[
\int_0^t \lal \pa_\tau w-\Delta w,\Phi\ral\,d\tau = \int_0^t \lal \nabla\cdot F+\nabla\otimes\nabla:G+\pa_\tau H,\Phi\ral\,d\tau .
\]
Using the backward heat equation for $\Phi$, integration by parts in time, and spatial integration by parts, we get
\begin{align*}
\lal w(t),\phi\ral &=\lal w_0,\mh_t\star_x\phi\ral +\int_0^t \lal F(\tau),\nabla(\mh_{t-\tau}\star_x\phi)\ral\,d\tau   +\int_0^t \lal G(\tau),\nabla\otimes\nabla(\mh_{t-\tau}\star_x\phi)\ral\,d\tau \\
&\quad +\lal H(t),\phi\ral -\lal H(0),\mh_t\star_x\phi\ral  +\int_0^t \lal H(\tau),\Delta(\mh_{t-\tau}\star_x\phi)\ral\,d\tau .
\end{align*}
This proves \eqref{eq:para-rep} for $\phi\in C_c^\infty(\R^2)$. The general case $\phi\in H^{s+1}(\R^2)$ follows by density.
\end{proof}

%
%
%
%
%
%

\subsubsection{Negative Sobolev estimate for $v^\e-v$}

We first consider the velocity difference
\[
w_v:=v^\e-v.
\]
Recall that, since $(f^\e,v^\e)$ is a weak entropy solution to \eqref{main1}, the fluid equation is satisfied in the sense of distributions. In particular, using the kinetic momentum balance, we may rewrite the fluid equation equivalently as
\[
\pa_t v^\e+v^\e\cdot\nabla v^\e+\nabla P^\e-\Delta v^\e = -\e\pa_t(\rho^\e u^\e)-\nabla\cdot\lt(\intr \xi\otimes\xi f^\e\,d\xi\rt)
\]
in $\md'((0,T)\times\R^2)$, while $v$ satisfies
\[
\pa_t v+v\cdot\nabla v+\nabla P-\Delta v=0, \quad \nabla\cdot v=0.
\]
Applying the Leray projection operator $\mathbb{P} := {\rm id} +\nabla (-\Delta)^{-1}\nabla \cdot$, we remove the pressure terms and obtain
\[
\pa_t w_v-\Delta w_v = -\mathbb P\nabla\cdot (v^\e\otimes v^\e-v\otimes v) -\e\pa_t(\rho^\e u^\e) -\mathbb P\nabla\cdot\lt(\intr (\xi\otimes\xi-\mathbb{I}_2)f^\e\,d\xi\rt).
\]

Recall that for $s\ge0$, the dual of $H^s(\R^2)$ is $H^{-s}(\R^2)$, and that
\[
L^1(\R^2)\hookrightarrow H^{-s}(\R^2) \quad\text{for }s>1.
\]
This weak topology is convenient for the parabolic representation formula derived above. In particular, since $v^\e,v\in L^2(\R^2)$, the quantity $v^\e-v$ is naturally viewed as an element of $H^{-s_1}(\R^2)$ for any $s_1>1$.
Applying Lemma \ref{lem:para-rep} componentwise, and using that $\mh_{t-s}\star_x$ commutes with the Leray projector $\mathbb P$, we obtain, for any $\psi\in H^{s_1}(\R^2)$ and $0<t\le T$,
\begin{align*}
\int_{\R^2} (v^\e-v)\cdot\psi\,dx
&=\int_{\R^2}\mh_t\star_x(v_0^\e-v_0)\cdot\psi\,dx  + \int_0^t \intr (v^\e\otimes v^\e-v\otimes v)(s):
\mathbb P\nabla(\mh_{t-s}\star_x\psi)\,dxds\\
&\quad -\e\int_0^t \intr \rho^\e u^\e(s)\cdot
\Delta(\mh_{t-s}\star_x\mathbb P\psi)\,dxds -\e\intr \rho^\e u^\e(t)\cdot \mathbb P\psi\,dx\\
&\quad 
+\e\intr \mh_t\star_x(\rho_0^\e u_0^\e)\cdot\mathbb P\psi\,dx +\int_0^t \intr
\lt(\intr (\xi\otimes\xi-\mathbb{I}_2)f^\e\,d\xi\rt)(s):
\nabla(\mh_{t-s}\star_x\mathbb P\psi)\,dxds\\
&=: \sum_{i=1}^6 I_i.
\end{align*}
We now estimate the six terms on the right-hand side.

For the initial contribution $I_1$, we simply use duality to obtain
\[
I_1 \le \|v_0^\e - v_0\|_{H^{-s_1}} \|\mh_t \star_x \psi\|_{H^{s_1}} \le \|\psi\|_{H^{s_1}} \|v_0^\e -v_0\|_{H^{-s_1}}.
\]
To estimate the nonlinear convection term $I_2$, we use Plancherel's theorem together with the standard heat-kernel bounds
\[
\|\mh_t \star f\|_{L^q} \le C t^{- \lt( \frac1p -\frac1q\rt)} \|f\|_{L^p}, \quad \|\nabla\mh_t\star f\|_{L^q} \le C t^{- \lt( \frac1p -\frac1q\rt) -\frac12} \|f\|_{L^p}, \quad 1 \le p \le q \le \infty,
\]
to have
\begin{align*}
I_2 &= \int_0^t \intr \lt[(v^\e - v)\otimes (v^\e -v)(s) + v\otimes (v^\e - v) + (v^\e - v)\otimes v \rt] : \mh_{t-s}\star_x \mathbb{P}\nabla\psi\,dxds\\
&\le \int_0^t \|v^\e -v\|_{L^4}^2 \|\mh_{t-s}\star_x \mathbb{P}\nabla\psi\|_{L^2}\,ds + \int_0^t \|v^\e -v\|_{H^{-s_1}} \|v (\mh_{t-s}\star_x\mathbb{P}\nabla\psi)\|_{H^{s_1}}\,ds\\
&\le C\|\nabla\psi\|_{L^2} \int_0^t \|v^\e -v\|_{L^2} \|\nabla(v^\e -v)\|_{L^2}\,ds + C\|v\|_{L^\infty(0,T;H^{s_1})} \int_0^t \|v^\e - v\|_{H^{-s_1}}\|\mh_{t-s}\star\mathbb{P}\nabla\psi\|_{H^{s_1}}\,ds\\
&\le C\|\psi\|_{H^{s_1}} \lt(\mathscr{H}[f^\e | M_{\rho, {\rm u}_\e}](0)  + \|v^\e_0 - v_0\|_{L^2}^2 + \e^2\rt)\\
&\quad + C\|v\|_{L^\infty(0,T;H^{s_1})} \|\psi\|_{H^{s_1}}\int_0^t (t-s)^{-\frac12}\|v^\e - v\|_{H^{-s_1}}\,ds,
\end{align*}
where we used
\[
x^k e^{-ax} \le \lt(\frac{k}{a} \rt)^k e^{-k}\quad \mbox{on }  x >0
\]
and
\[
\|\mh_{t-s}\star\mathbb{P}\nabla\psi\|_{H^{s_1}}^2  \le \intr (1+|\eta|^2)^{s_1} |\eta|^2 e^{-2(t-s)|\eta|^2} |\hat\psi|^2\,d\eta \le C(t-s)^{-1} \|\psi\|_{H^{s_1}}^2.
\]

For the time-derivative contribution $I_3$, we use the estimate 
\[
\intr \rho^\e|u^\e -{\rm u}_\e|^2\,dx \le 2\mathscr{H}[f^\e |M_{\rho, {\rm u}_\e}] 
\]
and get
\begin{align*}
I_3 &= -\e \int_0^t \intr \rho^\e u^\e \mh_{t-s}\star_x \Delta \mathbb{P} \psi\,dxds\\
&\le C\e \|\psi\|_{H^{s_1}}\int_0^t (t-s)^{-\frac{3-s_1}{2}} \lt(\|\rho^\e (u^\e - {\rm u}_\e)\|_{L^1} + \|\rho^\e {\rm u}_\e\|_{L^1}\rt)\,ds\\
&\le C\e \|\psi\|_{H^{s_1}} \lt(\int_0^t (t-s)^{-\frac{3-s_1}{2}}\,ds\rt)^{\frac12} \lt(\int_0^t \intr (t-s)^{-\frac{3-s_1}{2}}  \rho^\e |u^\e - {\rm u}_\e|^2\,dxds\rt)^{\frac12} +C\e^2 \|\psi\|_{H^{s_1}}\\
&\le C\e \|\psi\|_{H^{s_1}} \lt( \mathscr{H}[f^\e | M_{\rho, {\rm u}_\e}](0)  + \|v^\e_0 - v_0\|_{L^2}^2 + \e^2\rt)^{\frac12},
\end{align*}
where we used
\begin{align*}
\|\mh_{t-s}\star_x \Delta \mathbb{P} \psi\|_{L^\infty} &\le \intr |\eta|^2 e^{-(t-s)|\eta|^2}|\hat\psi|\,d\eta\\
&\le C\lt(\intr |\eta|^{4-2s_1} e^{-(t-s)|\eta|^2}\,d\eta \rt)^{\frac12}\|\psi\|_{H^{s_1}}\\
&\le C(t-s)^{-\frac{3-s_1}{2}} \|\psi\|_{H^{s_1}}.
\end{align*}

The endpoint term $I_4$ is bounded similarly:
\begin{align*}
I_4 &\le C\e\|\rho^\e u^\e\|_{L^1} \|\mathbb{P}\psi\|_{L^\infty}\\
&\le C\e \lt( \|\rho^\e (u^\e -{\rm u}_\e)\|_{L^1} + \e\rt) \|\mathbb{P}\psi\|_{H^{s_1}}\\
&\le C\e \|\psi\|_{H^{s_1}} \lt( \mathscr{H}[f^\e | M_{\rho, {\rm u}_\e}](0)  + \|v^\e_0 - v_0\|_{L^2}^2 + \e^2\rt)^{\frac12}.
\end{align*}

For the initial momentum term $I_5$, we directly deduce
\[
I_5  \le \e \|\rho_0^\e u_0^\e\|_{L^1} \|\mh_t \star_x \mathbb{P}\psi\|_{L^\infty} \le C\e \|\psi\|_{H^{s_1}} \lt(  \mathscr{H}[f^\e | M_{\rho, {\rm u}_\e}](0) + \e^2\rt)^{\frac12}.
\]

Finally, to estimate the stress contribution $I_6$, we first observe that
\begin{align*}
\intr (\xi\otimes \xi -\mathbb{I}_2)f^\e\,d\xi &= \intr (\xi \otimes \xi - u^\e \otimes u^\e - \mathbb{I}_2)f^\e\,d\xi + \rho^\e u^\e \otimes u^\e\\
&= \intr \lt[((\xi - u^\e)f^\e +\nabla_\xi f^\e)\otimes \xi + u^\e \otimes ((\xi -u^\e)f^\e +\nabla_\xi f^\e ) \rt]\,d\xi + \rho^\e u^\e \otimes u^\e.
\end{align*}
This implies
\begin{align*}
I_6 &\le \int_0^t \lt(\iint_{\R^2\times\R^2} \frac{1}{f^\e}|(\xi-u^\e) f^\e + \nabla_\xi f^\e|^2(s)\, dxd\xi\rt)^{\frac12}\lt(\iint_{\R^2\times\R^2} |\xi|^2 f^\e(s) \,dxd\xi \rt)^{\frac12} \|\mh_{t-s} \star_x \nabla  \mathbb{P}\psi\|_{L_x^\infty}\,ds\\
&\quad + \int_0^t \lt(\iint_{\R^2\times\R^2} \frac{1}{f^\e}|(\xi-u^\e) f^\e + \nabla_\xi f^\e|^2(s)\, dxd\xi\rt)^{\frac12}\lt(\intr \rho^\e  |u^\e|^2 (s) \,dx \rt)^{\frac12} \|\mh_{t-s} \star_x \nabla \mathbb{P}\psi\|_{L_x^\infty}\,ds\\
&\quad + \int_0^t \intr \lt(\rho^\e |u^\e - {\rm u}_\e|^2 + \rho^\e |{\rm u}_\e|^2\rt) \|\mh_{t-s} \star_x \nabla \mathbb{P} \psi\|_{L_x^\infty}\,dxds\\
&\le C\|\psi\|_{H^{s_1}} \lt(\int_0^t (t-s)^{-\frac12}\lt(\iint_{\R^2\times\R^2}\frac{1}{f^\e}|(\xi-u^\e) f^\e + \nabla_\xi f^\e|^2(s)\,dxd\xi\rt)\,ds\rt)^{\frac12}\\
&\quad + C\|\psi\|_{H^{s_1}} \int_0^t (t-s)^{-\frac12}\lt(\mathscr{H}[f^\e |M_{\rho, {\rm u}_\e}](s) + \e^2 \rt)\,ds\\
&\le C\|\psi\|_{H^{s_1}} \lt(\int_0^t (t-s)^{-\frac12}\lt(\iint_{\R^2\times\R^2} \frac{1}{f^\e}|(\xi-u^\e) f^\e + \nabla_\xi f^\e|^2(s)\,dxd\xi\rt) ds\rt)^{\frac12}\\
&\quad + C\|\psi\|_{H^{s_1}} \lt(\mathscr{H}[f^\e | M_{\rho, {\rm u}_\e}](0)  + \|v^\e_0 - v_0\|_{L^2}^2 + \e^2\rt).
\end{align*}

We collect all the estimates for $I_i$'s to yield
\begin{align*}
 \|(v^\e - v)(t)\|_{H^{-s_1}} &\le C\|v_0^\e - v_0\|_{H^{-s_1}}  + C \lt(\mathscr{H}[f^\e | M_{\rho, {\rm u}_\e}](0)  + \|v^\e_0 - v_0\|_{L^2}^2 + \e^2\rt)\\
&\quad + C\int_0^t (t-s)^{-\frac12}\|(v^\e -v)(s)\|_{H^{-s_1}}\,ds\\
&\quad + C \lt(\int_0^t (t-s)^{-\frac12}\lt(\iint_{\R^2\times\R^2} \frac{1}{f^\e}|(\xi-u^\e) f^\e + \nabla_\xi f^\e|^2(s)\,dxd\xi\rt) ds\rt)^{\frac12}.
\end{align*}

To close the above estimate, we recall the following lemma.
\begin{lemma}\cite[Lemma 4.2]{CJpre}\label{lem_aux_diff}
Let $T>0$, $\alpha>0$, and let $f,g \in L^2(0,T)$ be nonnegative functions such that
\[
f(t) \le g(t) + C\int_0^t \bigl((t-s)^{-\frac12} + (t-s)^{-1+\alpha}\bigr) f(s)\,ds
\]
for a.e. $t\in (0,T)$. Then there exists a constant $C_T>0$, depending only on $C$, $\alpha$, and $T$, such that
\[
\|f\|_{L^2(0,T)} \le C_T \|g\|_{L^2(0,T)}.
\]
\end{lemma}
Now, we write
\[\begin{aligned}
g(t) &:= C\|v_0^\e - v_0\|_{H^{-s_1}}  + C \lt(\mathscr{H}[f^\e | M_{\rho, {\rm u}_\e}](0)  + \|v^\e_0 - v_0\|_{L^2}^2 + \e^2\rt)\\
&\quad +  C \lt(\int_0^t (t-s)^{-\frac12}\lt(\iint_{\R^2\times\R^2} \frac{1}{f^\e}|(\xi-u^\e) f^\e + \nabla_\xi f^\e|^2(s)\,dxd\xi\rt) ds\rt)^{\frac12}.
\end{aligned}\]
Then one can check $g \in L^2(0,T)$. Indeed, by Fubini's theorem and \eqref{rel}, we obtain
\begin{align*}
\int_0^T |g(t)|^2\,dt
&\le C\lt(\|v_0^\e - v_0\|_{H^{-s_1}}
 + \mathscr{H}[f^\e | M_{\rho, {\rm u}_\e}](0) + \|v_0^\e - v_0\|_{L^2}^2
 + \e^2\rt)^2\\
&\quad + C\int_0^T \int_0^t (t-s)^{-\frac12}
\lt(\iint_{\R^2\times\R^2}\frac{1}{f^\e}|(\xi-u^\e) f^\e + \nabla_\xi f^\e|^2(s)\,dxd\xi\rt) dsdt\\
&=  C\lt(\|v_0^\e - v_0\|_{H^{-s_1}}
 + \mathscr{H}[f^\e | M_{\rho, {\rm u}_\e}](0) + \|v_0^\e - v_0\|_{L^2}^2
 + \e^2\rt)^2\\
&\quad + C\int_0^T
\lt(\iint_{\R^2\times\R^2}\frac{1}{f^\e}|(\xi-u^\e) f^\e + \nabla_\xi f^\e|^2(s)\,dxd\xi\rt)
\lt(\int_s^T (t-s)^{-\frac12}\,dt\rt)ds\\
&\le C\lt(\|v_0^\e - v_0\|_{H^{-s_1}}
 + \mathscr{H}[f^\e | M_{\rho, {\rm u}_\e}](0) + \|v_0^\e - v_0\|_{L^2}^2
 + \e^2\rt)^2\\
&\quad + C\int_0^T\iint_{\R^2\times\R^2}\frac{1}{f^\e}|(\xi-u^\e) f^\e + \nabla_\xi f^\e|^2\,dxd\xi dt\\
&\le C\lt(\|v_0^\e - v_0\|_{H^{-s_1}}
 + \mathscr{H}[f^\e | M_{\rho, {\rm u}_\e}](0) + \|v_0^\e - v_0\|_{L^2}^2
 + \e^2\rt)^2,
\end{align*}
Thus, we deduce from Lemma \ref{lem_aux_diff} that
\bq\label{est_v_hs1}
\|v^\e-v\|_{L^2(0,T;H^{-s_1})} \le C\lt( \|v_0^\e - v_0\|_{H^{-s_1}} + \mathscr{H}[f^\e | M_{\rho, {\rm u}_\e}](0)  + \|v^\e_0 - v_0\|_{L^2}^2 + \e^2\rt).  
\eq
This together with Lemma \ref{lem:equi} concludes the desired negative Sobolev estimate for the velocity error.

%
%
%
%
%
%

\subsubsection{Negative Sobolev estimate for $\rho^\e - \rho$}
We next turn to the density difference
\[
w_\rho:=\rho^\e-\rho.
\]
Since $(f^\e,v^\e)$ is a weak entropy solution to \eqref{main1}, the corresponding density equation is satisfied in the sense of distributions. In particular, using the local mass and momentum balances, we may rewrite the equation for $\rho^\e$ as
\[
\pa_t \rho^\e-\Delta \rho^\e = -\nabla\cdot(\rho^\e v^\e) +\nabla\cdot\lt(\e \pa_t(\rho^\e u^\e)\rt) +\nabla\otimes\nabla:\lt(\intr (\xi\otimes\xi-\mathbb{I}_2)f^\e\,d\xi\rt)
\]
in $\md'((0,T)\times\R^2)$, while $\rho$ solves
\[
\pa_t \rho-\Delta \rho = -\nabla\cdot(\rho v).
\]
Therefore, the difference $w_\rho=\rho^\e-\rho$ satisfies
\[
\pa_t w_\rho-\Delta w_\rho = -\nabla\cdot(\rho^\e v^\e-\rho v) +\nabla\cdot\lt(\e \pa_t(\rho^\e u^\e)\rt) +\nabla\otimes\nabla:\lt(\intr (\xi\otimes\xi-\mathbb{I}_2)f^\e\,d\xi\rt).
\]

Before applying the parabolic representation formula, we first note that
\[
\rho^\e \in L^\infty(0,T;H^{-1}(\R^2)).
\]
Indeed, we use Lemmas \ref{lem:TM} and \ref{lem:mom} to attain, for any $\varphi \in H^1(\R^2)$,
\[
\lt|\intr \rho^\e \varphi\,dx\rt| \le \lt(\intr \rho^\e \,dx\rt)^{\frac12} \lt(\intr \rho^\e |\varphi|^2\,dx\rt)^{\frac12} \le C\lt(1+ \sup_{0 <t<T} \intr \rho^\e |\log \rho^\e|\,dx \rt)^{\frac12} \|\varphi\|_{H^1}.
\]
Thus, $L(\varphi) := \intr \rho^\e \varphi\,dx$ is a bounded linear functional on $H^1(\R^2)$ and this implies $\rho^\e \in L^\infty(0,T;H^{-1}(\R^2))$. This also implies $\rho^\e \in L^\infty(0,T;H^{-s}(\R^2))$ for any $s>1$.

We apply Lemma \ref{lem:para-rep} to each component of the vector field $\e\pa_t(\rho^\e u^\e)$ and to the tensor field
\[
\intr (\xi\otimes\xi-\mathbb{I}_2)f^\e\,d\xi.
\]
Then, for any $\varphi\in H^{s_1+1}(\R^2)$ and $0<t\le T$, we obtain
\begin{align*}
\int_{\R^2} (\rho^\e-\rho)\varphi\,dx
&=\int_{\R^2}\mh_t\star_x(\rho_0^\e-\rho_0)\,\varphi\,dx +\int_0^t \intr (\rho^\e v^\e-\rho v)(s)\cdot
\nabla(\mh_{t-s}\star_x\varphi)\,dxds\\
&\quad -\e\int_0^t \intr \rho^\e u^\e(s)\cdot \nabla\Delta(\mh_{t-s}\star_x\varphi)\,dxds -\e\intr \rho^\e u^\e(t)\cdot\nabla\varphi\,dx \\
&\quad 
+\e\intr \mh_t\star_x(\rho_0^\e u_0^\e)\cdot\nabla\varphi\,dx  +\int_0^t \intr
\lt(\intr (\xi\otimes\xi-\mathbb{I}_2)f^\e(s)\,d\xi\rt):
\nabla\otimes\nabla(\mh_{t-s}\star_x\varphi)\,dxds\\
&=: \sum_{i=1}^6 J_i.
\end{align*}
We again estimate the terms on the right-hand side separately.

For the initial contribution $J_1$, duality immediately gives
\[
J_1\le \|\rho_0^\e - \rho_0\|_{H^{-s_1-1}}\|\mh_t \star_x \varphi\|_{H^{s_1+1}}\le \|\varphi\|_{H^{s_1+1}}\|\rho_0^\e - \rho_0\|_{H^{-s_1-1}}.
\]

To estimate the transport term $J_2$, we first observe that 
\[
\max\{a,b\}\min\{a^{-1}, b^{-1}\} = 1 
\]
holds for any $a, b>0$. Together with Lemma \ref{lem:equi}, this yields
\begin{align*}
\intr \lt| (\rho^\e - \rho)(v^\e -v)\rt|\,dx
&\le \lt(\intr \min\{(\rho^\e)^{-1}, \rho^{-1}\} |\rho^\e -\rho|^2\,dx\rt)^{\frac12}\lt(\intr \max\{\rho^\e, \rho\} |v^\e - v|^2\,dx\rt)^{\frac12}\\
&\le C\lt( \intr \rho^\e \log\lt(\frac{\rho^\e}{\rho} \rt)\,dx\rt)^{\frac12} \lt(\intr (\rho^\e + \rho)|v^\e -v|^2\,dx \rt)^{\frac12}\\
&\le C\mathscr{H}[f^\e |M_{\rho, {\rm u}_\e}] + C\|v^\e -v\|_{H^1}^2.
\end{align*}
We then decompose the transport difference as
\[
\rho^\e v^\e-\rho v
=(\rho^\e-\rho)(v^\e-v)+(\rho^\e-\rho)v+\rho(v^\e-v),
\]
and estimate the three terms separately according to the relative entropy bound, the already established estimate for $v^\e-v$, and the regularity assumptions on $\rho$ and $v$. Indeed, we obtain
\begin{align*}
J_2 &\le \int_0^t \intr |(\rho^\e -\rho)(v^\e -v)(s)| \|\mh_{t-s}\star_x \nabla\varphi\|_{L_x^\infty}\,dxds\\
&\quad + \int_0^t \|(\rho^\e -\rho)(s)\|_{H^{-s_1-1}} \|v(\mh_{t-s}\star_x \nabla\varphi)\|_{H^{s_1+1}}\,ds\\
&\quad + \int_0^t \|(v^\e - v)(s)\|_{H^{-s_1}}\|\rho (\mh_{t-s}\star_x \nabla\varphi)\|_{H^{s_1}}\,ds\\
&\le C \int_0^t \intr |(\rho^\e -\rho)(v^\e -v)(s)| \|\mh_{t-s}\star_x \varphi\|_{H^{s_1+1}}\,dxds\\
&\quad + C\|v\|_{L^\infty(0,T;H^{s_1+1})}\int_0^t \|(\rho^\e -\rho)(s)\|_{H^{-s_1-1}}\|\mh_{t-s}\star_x \nabla\varphi\|_{H^{s_1+1}}\,ds\\
&\quad + C\|\rho\|_{L^\infty(0,T;H^{s_1})}\int_0^t\|(v^\e -v)(s)\|_{H^{-s_1}} \|\mh_{t-s}\star_x\nabla\varphi\|_{H^{s_1}}\,ds\\
&\le C\|\varphi\|_{H^{s_1+1}} \int_0^t \lt(\mathscr{H}[f^\e |M_{\rho, {\rm u}_\e}](s) + \|(v^\e - v)(s)\|_{L^2}^2 + \|\nabla (v^\e -v)(s)\|_{L^2}^2\rt)\,ds\\
&\quad + C\|v\|_{L^\infty(0,T;H^{s_1+1})}\int_0^t (t-s)^{-\frac12} \|(\rho^\e -\rho)(s)\|_{H^{-s_1-1}}\,ds\\
&\quad + C\|\rho\|_{L^\infty(0,T;H^{s_1})}\|\varphi\|_{H^{s_1+1}}\int_0^t \|(v^\e -v)(s)\|_{H^{-s_1}}\,ds\\
&\le C\lt(\mathscr{H}[f^\e|M_{\rho, {\rm u}_\e}](0) + \|v_0^\e - v_0\|_{L^2}^2 + \e^2 \rt)\|\varphi\|_{H^{s_1+1}}\\
&\quad +  C\|v\|_{L^\infty(0,T;H^{s_1+1})}\int_0^t (t-s)^{-\frac12} \|(\rho^\e -\rho)(s)\|_{H^{-s_1-1}}\,ds\\
&\quad + C\|\rho\|_{L^\infty(0,T;H^{s_1})}\|\varphi\|_{H^{s_1+1}}\lt(\int_0^t \|(v^\e -v)(s)\|_{H^{-s_1}}^2\,ds\rt)^{\frac12}.
\end{align*}

The terms $J_3$, $J_4$, and $J_5$ are handled exactly as in the estimates of $I_3$, $I_4$, and $I_5$, respectively:
\begin{align*}
J_3 &\le C\e \|\varphi\|_{H^{s_1+1}}\lt(\mathscr{H}[f^\e |M_{\rho, {\rm u}_\e}](0) + \|v_0^\e - v_0\|_{L^2}^2 + \e^2 \rt)^{\frac12},\\
J_4 &\le C\e \|\varphi\|_{H^{s_1+1}}\lt(\mathscr{H}[f^\e |M_{\rho, {\rm u}_\e}](0) + \|v_0^\e - v_0\|_{L^2}^2 + \e^2 \rt)^{\frac12},\\
J_5 &\le C\e \|\varphi\|_{H^{s_1+1}}\lt(\mathscr{H}[f^\e |M_{\rho, {\rm u}_\e}](0) + \e^2 \rt)^{\frac12}.
\end{align*}

Finally, the stress term $J_6$ is controlled in the same way as $I_6$:
\begin{align*}
J_6 &\le C\|\varphi\|_{H^{s_1+1}}\lt(\int_0^t (t-s)^{-\frac12}\lt(\iint_{\R^2\times\R^2} \frac{1}{f^\e}|(\xi-u^\e)f^\e + \nabla_\xi f^\e|^2(s)\,dxd\xi\rt)ds \rt)^{\frac12}\\
&\quad + C\|\varphi\|_{H^{s_1+1}}\lt(\mathscr{H}[f^\e |M_{\rho, {\rm u}_\e}](0) + \|v_0^\e - v_0\|_{L^2}^2 + \e^2 \rt).
\end{align*}

Thus, collecting all the estimates for $J_i$'s, we obtain
\begin{align*}
\|(\rho^\e - \rho)(t)\|_{H^{-s_1-1}} &\le C\|\rho_0^\e - \rho_0\|_{H^{-s_1-1}} + C\lt(\mathscr{H}[f^\e |M_{\rho, {\rm u}_\e}](0) + \|v_0^\e - v_0\|_{L^2}^2 + \e^2 \rt)\\
&\quad + C\|v\|_{L^\infty(0,T;H^{s_1+1})}\int_0^t (t-s)^{-\frac12}\|(\rho^\e -\rho)(s)\|_{H^{-s_1-1}}\,ds\\
&\quad + C\|\rho\|_{L^\infty(0,T;H^{s_1})} \lt(\int_0^t \|(v^\e - v)(s)\|_{H^{-s_1}}^2\,ds\rt)^{\frac12}\\
&\quad + C\lt(\int_0^t (t-s)^{-\frac12}\lt(\iint_{\R^2\times\R^2} \frac{1}{f^\e}|(\xi-u^\e)f^\e + \nabla_\xi f^\e|^2(s)\,dxd\xi\rt)ds \rt)^{\frac12}.
\end{align*}
In this case, we write
\[
\begin{aligned}
g(t) &:= C\|\rho_0^\e - \rho_0\|_{H^{-s_1-1}} + C\lt(\mathscr{H}[f^\e |M_{\rho, {\rm u}_\e}](0) + \|v_0^\e - v_0\|_{L^2}^2 + \e^2 \rt)\\
&\quad + C\|\rho\|_{L^\infty(0,T;H^{s_1})} \lt(\int_0^t \|(v^\e - v)(s)\|_{H^{-s_1}}^2\,ds\rt)^{\frac12}\\
&\quad + C\lt(\int_0^t (t-s)^{-\frac12}\lt(\iint_{\R^2\times\R^2} \frac{1}{f^\e}|(\xi-u^\e)f^\e + \nabla_\xi f^\e|^2(s)\,dxd\xi\rt)ds \rt)^{\frac12}.
\end{aligned}
\]
Applying Lemma \ref{lem_aux_diff} as in the estimate for $\|v^\e -v\|_{L^2(0,T;H^{-s_1})}$, and then using \eqref{est_v_hs1}, we obtain the desired negative Sobolev estimate for the density error:
\[
\|\rho^\e -\rho\|_{L^2(0,T;H^{-s_1-1})}  \le C\lt(\|\rho_0^\e - \rho_0\|_{H^{-s_1-1}} + \|v_0^\e - v_0\|_{H^{-s_1}} + \mathscr{H}[f^\e |M_{\rho, 0}](0) + \|v_0^\e - v_0\|_{L^2}^2 + \e^2 \rt).
\]
This completes the proof.

%
%
%
%
%
%

\section*{Acknowledgments}
 The work of Y.-P. Choi was supported by NRF grant no. 2022R1A2C1002820 and RS-2024-00406821. The work of J. Jung was supported by NRF grant no. 2022R1A2C1002820.\\

\vspace{0.3cm}

\noindent {\bf Data availability.} This manuscript has no associated data.\\

\noindent {\bf Declaration - Conflict of interest.} The authors declare that they have no conflict of interest.

%
%
%
%
%
%

\appendix

\section{Formal diffusion limit and optimal expansion rates}\label{app:op_rate}

In this appendix, we briefly outline the formal diffusive scaling of the kinetic--fluid system \eqref{main1}, with the purpose of identifying the optimal expansion rates for $f^\e$, its macroscopic density
\[
\rho^\e(t,x):=\intr f^\e(t,x,\xi)\,d\xi,
\]
the kinetic momentum
\[
m^\e(t,x):=\intr \xi f^\e(t,x,\xi)\,d\xi=\rho^\e(t,x)u^\e(t,x),
\]
and the fluid velocity $v^\e$.

Throughout this appendix, we focus on the kinetic equation
\[
\pa_t f^\e + \frac1\e \xi\cdot\nabla_x f^\e
= \frac1{\e^2}\nabla_\xi\cdot(\nabla_\xi f^\e + (\xi-\e v^\e)f^\e),
\]
which can be rewritten as
\[
\pa_t f^\e + \frac1\e \xi\cdot\nabla_x f^\e
= \frac1{\e^2}\calL f^\e - \frac1\e \nabla_\xi\cdot(v^\e f^\e),
 \quad
\calL g:=\nabla_\xi\cdot(\nabla_\xi g+\xi g).
\]
The operator $\calL$ is the classical Fokker--Planck operator in velocity, whose kernel is spanned by the normalized Maxwellian
\[
M(\xi):=\frac1{2\pi}e^{-\frac{|\xi|^2}{2}}.
\]

\subsection{Hilbert expansion and first-order corrector}

We postulate the formal Hilbert expansions
\[
f^\e=f_0+\e f_1+\e^2 f_2+\cdots, \quad v^\e=v_0+\e v_1+\cdots,
\]
and compare powers of $\e$ in the kinetic equation.

At order $\e^{-2}$, we obtain $\calL f_0=0$, which implies
\[
f_0(t,x,\xi)=\rho(t,x)M(\xi)
\]
for some scalar field $\rho(t,x)$. At order $\e^{-1}$, we get
\[
\calL f_1=\xi\cdot\nabla_x(\rho M)+\nabla_\xi\cdot(v_0\rho M).
\]
Using $\nabla_\xi M=-\xi M$, we compute
\[
\xi\cdot\nabla_x(\rho M)=(\xi\cdot\nabla_x\rho)M, \quad \nabla_\xi\cdot(v_0\rho M)=\rho\,v_0\cdot\nabla_\xi M
=-\rho\,v_0\cdot\xi\,M.
\]
Hence
\[
\calL f_1=(\xi\cdot\nabla_x\rho-\rho\,v_0\cdot\xi)M.
\]
Recalling that
\[
\calL(\xi_i M)=-\xi_i M,
\]
we find
\bq\label{app_ide}
f_1(t,x,\xi)=(-\xi\cdot\nabla_x\rho+\rho\,v_0\cdot\xi)M.
\eq
The solvability condition is automatically satisfied at this order, since the right-hand side is odd in $\xi$ and hence orthogonal to $\ker\calL$. Moreover, this solution is uniquely determined in $(\ker\calL)^\perp$ by the coercivity of $\calL$; see, for instance, \cite{AMTU01}.

Integrating the kinetic equation in $\xi$, we obtain the exact continuity relation
\[
\pa_t\rho^\e+\frac1\e\nabla_x\cdot m^\e=0.
\]
Using \eqref{app_ide} and the Gaussian moment identities
\[
\intr \xi M\,d\xi=0, \quad \intr \xi_i\xi_j M\,d\xi=\delta_{ij},
\]
we find
\[
m^\e=\e(-\nabla_x\rho+\rho v_0)+O(\e^2).
\]
Substituting this expansion into the continuity equation yields, at leading order,
\[
\pa_t\rho+\nabla_x\cdot(-\nabla_x\rho+\rho v_0)=0.
\]
Since $\nabla_x\cdot v_0=0$, this can be written as
\[
\pa_t\rho+v_0\cdot\nabla_x\rho-\Delta_x\rho=0.
\]

\subsection{Formal optimal convergence rates}

We now extract the formal optimal rates suggested by the Hilbert expansion.

\medskip
\noindent\textbf{Rate for $f^\e$.}
From
\[
f^\e=f_0+\e f_1+\e^2 f_2+O(\e^3),
\]
together with the explicit forms of $f_0$ and $f_1$, we obtain
\[
f^\e = \rho M+\e(-\xi\cdot\nabla_x\rho+\rho\,v_0\cdot\xi)M+O(\e^2).
\]
In particular, truncating at order $0$ gives the weaker approximation
\[
f^\e=\rho M+O(\e).
\]

\medskip
\noindent\textbf{Rate for $\rho^\e$.}
Integrating the Hilbert expansion in $\xi$, we find
\[
\rho^\e = \intr f_0\,d\xi+\e\intr f_1\,d\xi+\e^2\intr f_2\,d\xi+O(\e^3).
\]
Since
\[
\intr f_0\,d\xi=\rho, \quad \intr f_1\,d\xi=0
\]
by oddness in $\xi$, the first-order correction vanishes. Therefore
\[
\rho^\e-\rho = \e^2\intr f_2\,d\xi+O(\e^3),
\]
and hence
\[
\rho^\e=\rho+O(\e^2).
\]

\medskip
\noindent\textbf{Rate for the kinetic momentum.}
Expanding $m^\e$ gives
\[
m^\e =\intr \xi f_0\,d\xi+\e\intr \xi f_1\,d\xi+\e^2\intr \xi f_2\,d\xi+O(\e^3).
\]
Since $\intr \xi f_0\,d\xi=0$, we only need the first-order term. Using \eqref{app_ide}, we obtain
\[
\intr \xi f_1\,d\xi=-\nabla_x\rho+\rho v_0.
\]
Consequently,
\[
m^\e=\e(-\nabla_x\rho+\rho v_0)+O(\e^2),
\]
or equivalently,
\[
\frac1\e m^\e
=\frac1\e \rho^\e u^\e
=-\nabla_x\rho+\rho v_0+O(\e).
\]
Hence the leading-order kinetic momentum is of order $\e$, while the rescaled momentum admits a nontrivial limit.

\medskip
\noindent\textbf{Second-order cancellation in the fluid forcing.}
We next show that the forcing term in the fluid equation exhibits an additional cancellation at order $\e$.
Let
\[
\Sigma^\e:=\intr \xi\otimes\xi\,f^\e\,d\xi.
\]
Multiplying the kinetic equation by $\xi$ and integrating in $\xi$, we obtain the exact momentum balance
\[
\pa_t m^\e+\frac1\e \nabla_x\cdot\Sigma^\e
=-\frac1{\e^2}m^\e+\frac1\e \rho^\e v^\e.
\]
We expand
\[
m^\e=\e m_1+\e^2 m_2+\e^3 m_3+\cdots, \quad \Sigma^\e=\Sigma_0+\e\Sigma_1+\e^2\Sigma_2+\cdots.
\]
Since
\[
f_0=\rho M, \quad f_1=(-\xi\cdot\nabla_x\rho+\rho\,v_0\cdot\xi)M,
\]
we have
\[
\Sigma_0=\intr \xi\otimes\xi\,f_0\,d\xi=\rho \mathbb{I}_2,
\]
while
\[
\Sigma_1=\intr \xi\otimes\xi\,f_1\,d\xi=0
\]
by oddness of the third-order Gaussian moments. Comparing the coefficients of $\e^{-1}$ in the momentum equation yields
\[
m_1=-\nabla_x\rho+\rho v_0,
\]
whereas the coefficients of $\e^0$ give
\[
\nabla_x\cdot\Sigma_1=-m_2+\rho v_1.
\]
Since $\Sigma_1=0$, we infer that
\[
m_2=\rho v_1.
\]

Therefore, the fluid forcing admits the refined expansion
\[
\frac1\e \intr (\xi-\e v^\e)f^\e\,d\xi =\frac1\e m^\e-\rho^\e v^\e =(m_1-\rho v_0)+\e(m_2-\rho v_1)+O(\e^2).
\]
Using the above identities, we conclude that
\[
\frac1\e \intr (\xi-\e v^\e)f^\e\,d\xi =-\nabla_x\rho+O(\e^2).
\]

\medskip
\noindent\textbf{Limit equation and second-order approximation for $v^\e$.}
At leading order, the fluid equation becomes
\[
\pa_t v_0+v_0\cdot\nabla_x v_0+\nabla_x P_0-\Delta_x v_0=-\nabla_x\rho, \quad \nabla_x\cdot v_0=0.
\]
Since the forcing term $-\nabla_x\rho$ is a pure gradient, it can be absorbed into the pressure. Defining
\[
\Pi:=P_0+\rho,
\]
we obtain
\[
\pa_t v_0+v_0\cdot\nabla_x v_0+\nabla_x \Pi-\Delta_x v_0=0, \quad \nabla_x\cdot v_0=0.
\]
Thus, the limit velocity field $v:=v_0$ is governed by the incompressible Navier--Stokes equations.

Moreover, since the forcing term has no contribution at order $\e$, the coefficient $v_1$ satisfies
\[
\pa_t v_1+v_0\cdot\nabla_x v_1+v_1\cdot\nabla_x v_0+\nabla_x P_1-\Delta_x v_1=0, \quad \nabla_x\cdot v_1=0.
\]
Hence, under well-prepared initial data satisfying
\[
v^\e(0)=v_0(0)+\e v_1(0) + O(\e^2), \quad \mbox{with} \quad v_1(0) = 0,
\]
the first-order corrector vanishes, that is,
\[
v_1\equiv 0,
\]
and therefore
\[
v^\e-v=O(\e^2)
\]
at the formal level.

In summary, the Hilbert expansion predicts the formal asymptotic rates
\[
f^\e=\rho M+O(\e), \quad \rho^\e=\rho+O(\e^2),
\]
and, more precisely,
\[
f^\e =\rho M+\e(-\xi\cdot\nabla_x\rho+\rho\,v\cdot\xi)M+O(\e^2),
\]
together with
\[
\frac1\e \rho^\e u^\e =-\nabla_x\rho+\rho v+O(\e).
\]
Moreover, the forcing term in the fluid equation satisfies the refined expansion
\[
\frac1\e \intr (\xi-\e v^\e)f^\e\,d\xi =-\nabla_x\rho+O(\e^2).
\]
In particular, under a well-prepared expansion of the initial data, the first-order corrector in the fluid velocity vanishes, and the formal expansion yields
\[
v^\e-v=O(\e^2).
\]

%
%
%
%
%
%

\section{Existence theory for the limiting system}\label{app:ext}
In this part, we discuss the existence of solutions to the limiting system \eqref{lim_light} satisfying the required regularity stated in Theorem \ref{thm_main}. Recall that the limiting system reads
\begin{align}\label{lim_light2}
\begin{aligned}
&\pa_t \rho +  \nabla  \cdot (\rho v) =   \Delta  \rho, \cr
&\pa_t v + v\cdot \nabla  v + \nabla P -  \Delta  v = 0,\cr
&\nabla \cdot v = 0.
\end{aligned}
\end{align}
Since the incompressible Navier--Stokes system is not affected by the advection--diffusion equation for $\rho$, we focus on the diffusion equation. More precisely, we provide the propagation of regularity for $\rho$ for a given sufficiently smooth $v$.
 
Define a solution space $X_{s,p}(T)$ for $p>2$ and $T>0$:
\[
X_{s,p}(T) := \{ \rho > 0 :  \rho \in L^\infty([0,T]; L^1 \cap L^\infty(\om)), \quad  \nabla \log \rho \in L^\infty([0,T]; (L^p \cap \dot{H}^s)(\om)) \}.
\]
\begin{theorem}
Let $s>2$ and let $T>0$ be arbitrary. Let $v$ be a given divergence-free vector field on $[0,T]\times\Omega$
such that
\[
\int_0^T \|v(t)\|_{W^{1,\infty}\cap \dot{H}^{s+1}}(\Omega)\,dt <\infty.
\]
Assume $\rho_0 > 0$, $\rho_0\in L^1(\Omega)\cap L^\infty(\Omega)$, and
$\phi_0:=\nabla\log\rho_0\in (L^p\cap \dot{H}^s)(\Omega)$.
\begin{itemize}
\item[(i)] (Local existence and continuation criterion) There exists $T_*\in(0,T]$, depending on $\|\phi_0\|_{L^p\cap \dot{H}^s}$ and $\|\nabla v\|_{L^1(0,T;L^\infty)}$, and a unique solution $\rho\in X_{s,p}(T_*)$ of \eqref{lim_light2}. Moreover, this solution can be continued up to any time $T'\le T$ provided that
\[
\int_0^{T'} \|\nabla\phi(t)\|_{L^\infty(\Omega)}\,dt <\infty.
\]

\item[(ii)] (Global-in-time existence on $\T^2$) If $\Omega=\T^2$, then the above continuation criterion is satisfied for all $T'>0$, and hence $\rho\in X_{s,p}(T)$.
\end{itemize}
\end{theorem}

\begin{remark}
Note that if $\om=\T^2$, we may assume $\phi_0 \in H^s(\T^2)$, which can be allowed by the positive infimum of $\rho_0$. However, this is not the case for $\om = \R^2$. Instead, $\rho_0$ of the form
\[
\rho_0(x) := C_k (1+|x|^2)^{-k}
\]
satisfies the desired condition $\phi_0 \in (L^p \cap \dot{H}^s)(\om)$.
\end{remark}

\begin{proof}[Proof of Theorem B.1]
We argue at the level of smooth solutions (obtained by standard mollification of $\rho_0$ and $v$), derive a priori estimates independent of the mollification parameter, and then pass to the limit.

\medskip
\noindent\textbf{Step 1. A priori estimate for $\phi$ and a continuation criterion.}
We first observe that $\phi=\nabla\log\rho$ satisfies
\[
\pa_t\phi+\nabla(\phi\cdot v)=\Delta\phi+\nabla(|\phi|^2).
\]
 
We first estimate $L^p$-norm of $\phi$. For this,
\[\begin{aligned}
\frac1p\frac{d}{dt}\|\phi\|_{L^p}^p &= \intr |\phi|^{p-2}\phi \cdot (- \nabla (\phi \cdot v) +  \Delta \phi +  \nabla(|\phi|^2))\,dx\\
&=: I_1 + I_2 + I_3.
\end{aligned}\]
 For $I_1$, we use
\[
\pa_i( |\phi|^\ell) = \ell |\phi|^{\ell-2} \phi_j \pa_i \phi_j, \quad \ell \in \N, 
\]
to get
\[\begin{aligned}
I_1 &= -\intr  |\phi|^{p-2}\phi_i \pa_i (v_j \phi_j) \,dx = -\intr |\phi|^{p-2}\phi_i \phi_j \pa_i v_j \,dx - \intr |\phi|^{p-2} \phi_i \pa_i \phi_j v_j\,dx\\
&= -\intr |\phi|^{p-2}\phi_i \phi_j \pa_i v_j \,dx +\frac1p \intr \nabla (|\phi|^p) \cdot v \,dx\\
&\le \|\nabla v\|_{L^\infty}\|\phi\|_{L^p}^p.
\end{aligned}\]
For $I_2$ and $I_3$, we obtain
\[\begin{aligned}
I_2 = \intr |\phi|^{p-2} \phi_j \pa_{ii}\phi_j\,dx &= -\intr |\phi|^{p-2}|\nabla \phi|^2\,dx - (p-2)\intr |\phi|^{p-4}(\phi_k \pa_i \phi_k)(\phi_j \pa_i \phi_j)\,dx\\
&=   -\intr |\phi|^{p-2}|\nabla \phi|^2\,dx -\frac{4(p-2)}{p^2} \intr \left|\nabla |\phi|^{\frac p2} \right|^2 dx
\end{aligned}\]
and
\[\begin{aligned}
I_3 &= 2 \intr |\phi|^{p-2} \phi_j \pa_j \phi_k \phi_k \,dx = \frac2p \intr \nabla (|\phi|^p) \cdot \phi \,dx \le C\|\nabla\cdot\phi\|_{L^\infty}\|\phi\|_{L^p}^p.
\end{aligned}\]
Thus, we have
\[
\frac{d}{dt}\|\phi\|_{L^p}^p + p\intr |\phi|^{p-2}|\nabla \phi|^2\,dx +\frac{4(p-2)}{p} \intr \left|\nabla |\phi|^{\frac p2}  \right|^2 dx \le C\lt(\|\nabla v\|_{L^\infty} + \|\nabla\phi\|_{L^\infty} \rt)\|\phi\|_{L^p}^p.
\]
 
 Now, applying $\Lambda^a:= (-\Delta)^\frac a2$, $a > 0$ and testing against $\phi$ gives
\[\begin{aligned}
\frac12\frac{d}{dt}\|\Lambda^a \phi\|_{L^2}^2 + 2 \|\Lambda^a(\nabla  \phi)\|_{L^2}^2&= - \into \Lambda^a \nabla (\phi\cdot v) \cdot \Lambda^a \phi\,dx  + \into \Lambda^a (\nabla |\phi|^2) \cdot \Lambda^a \phi\,dx\\
&=: J_1 + J_2.
\end{aligned}\]
For $J_1$, we use the symmetry $\pa_j \phi_i = \pa_i \phi_j$ to estimate
\[\begin{aligned}
J_1 &= -\into v_j (\Lambda^a \pa_i \phi_j) \Lambda^a \phi_i\,dx - \into \lt[ \Lambda^a (\pa_i v_j \phi_j) - v_j \Lambda^a \pa_i \phi_j\rt] \Lambda^a \phi_i\,dx\\
&=-\into v_j (\Lambda^a \pa_j \phi_i) \Lambda^a \phi_i\,dx - \into \lt[ \Lambda^a (\pa_i v_j \phi_j) - v_j \Lambda^a \pa_i \phi_j\rt] \Lambda^a \phi_i\,dx\\
&= - \into \lt[ \Lambda^a (\pa_i v_j \phi_j) - v_j \Lambda^a \pa_i \phi_j\rt] \Lambda^a \phi_i\,dx\\
&\le C\|\nabla v\|_{L^\infty}\|\Lambda^a \phi\|_{L^2}^2   + C\|\Lambda^{a+1}v\|_{L^2}\|\phi\|_{L^\infty}\|\Lambda^a \phi\|_{L^2}.
\end{aligned}\]
The estimate of $J_2$ is similar:
\[\begin{aligned}
J_2 &= \into \phi_j (\Lambda^a \pa_i \phi_j)\Lambda^a \phi_i \,dx + \into \lt[\Lambda^a \pa_i (\phi_j \phi_j) -\phi_j (\Lambda^a \pa_i \phi_j)  \rt]\Lambda^a \phi_i \,dx\\
&= -\frac12 \into (\nabla\cdot\phi) |\Lambda^a \phi|^2\,dx + \into \lt[\Lambda^a \pa_i (\phi_j \phi_j) -\phi_j (\Lambda^a \pa_i \phi_j)  \rt]\Lambda^a \phi_i \,dx\\
&\le C\|\nabla\phi\|_{L^\infty}\|\Lambda^a \phi\|_{L^2}^2.
\end{aligned}\]
Consequently, we have
\bq\label{eq:Hs}
\frac{d}{dt}\|\phi\|_{L^p \cap \dot{H}^s}^2  \le C\lt(\|\nabla v\|_{L^\infty} + \|v\|_{\dot{H}^{s+1}} + \|\nabla\phi\|_{L^\infty}\rt)\|\phi\|_{L^p \cap \dot{H}^s}^2.
\eq
Since $s>2$ in two dimensions, $L^p \cap \dot{H}^s(\Omega)\hookrightarrow W^{1,\infty}(\Omega)$, hence $\|\nabla\phi\|_{L^\infty}\le C\|\phi\|_{L^p \cap \dot{H}^s}$ and \eqref{eq:Hs} implies a local-in-time bound
\[
\sup_{0\le t\le T_*}\|\phi(t)\|_{H^s}<\infty \quad\text{for some }T_*=T_*\lt(\|\phi_0\|_{H^s},\|\nabla v\|_{L^1(0,T;L^\infty)},   \|v\|_{L^1(0,T;\dot{H}^{s+1})}  \rt)>0.
\]
In particular, \eqref{eq:Hs} yields the following continuation criterion:
if
\bq\label{eq:cri}
\int_0^{T}\|\nabla\phi(t)\|_{L^\infty(\Omega)}\,dt<\infty,
\eq
then $\sup_{t\in[0,T]}\|\phi(t)\|_{H^s}<\infty$, and the solution can be continued up to time $T$.

\medskip
\noindent\textbf{Step 2. Local existence in $X_{s,p}(T_*)$.}
The bound on $\phi$ obtained above implies $\phi\in L^\infty(0,T_*;L^\infty(\om))$. Since $\rho$ solves $\pa_t\rho+v\cdot\nabla\rho=\Delta\rho$, the maximum principle gives $\rho\in L^\infty(0,T_*;L^1\cap L^\infty(\om))$, and the positivity $\rho>0$ is propagated by the maximum principle as long as the smooth solution exists. Therefore $\rho\in X_{s,p}(T_*)$. Uniqueness follows from the linearity of the advection--diffusion equation for $\rho$ with a given divergence-free velocity field $v\in L^1(0,T_*;W^{1,\infty}\cap \dot{H}^{s+1}(\om))$. In particular, two solutions with the same initial data must coincide, and hence $\phi=\nabla\log\rho$ is uniquely determined as well.

\medskip
\noindent\textbf{Step 3. Global extension on $\T^2$.}
Assume now $\Omega=\T^2$. Since $\phi_0=\nabla\log\rho_0\in H^s(\T^2)$ with $s>2$, we have $\rho_0\in C^1(\T^2)$ and $\rho_0>0$. By compactness of $\T^2$, there exists $c_0>0$ such that $\rho_0\ge c_0$ on $\T^2$. Moreover, since $\nabla\cdot v=0$, the equation can be written as $\pa_t\rho+v\cdot\nabla\rho=\Delta\rho$. The standard maximum principle then yields
\bq\label{eq:rho_mp}
c_0\le \rho(t,x)\le \|\rho_0\|_{L^\infty(\T^2)}
\quad\text{for all }(t,x)\in[0,T]\times\T^2.
\eq

Fix $t_0\in(0,T_*)$ and choose $\delta \in (0,\min\{T,T_*\}-t_0 )$. Writing $\pa_t\rho-\Delta\rho=-v\cdot\nabla\rho$, the Duhamel formula gives for $t\ge t_0$,
\bq\label{eq:duha}
\rho(t)=e^{(t-t_0)\Delta}\rho(t_0)-\int_{t_0}^t e^{(t-s)\Delta}\lt(v(s)\cdot\nabla\rho(s)\rt) ds.
\eq
Using the $L^\infty$-smoothing of the heat semigroup on $\T^2$,
\[
\|\nabla^m e^{\tau\Delta}f\|_{L^\infty}\le C_m \tau^{-m/2}\|f\|_{L^\infty},\quad \tau>0,
\]
we infer (for $t\ge t_0$)
\begin{align*}
\|\nabla\rho(t)\|_{L^\infty} &\le C(t-t_0)^{-\frac12}\|\rho(t_0)\|_{L^\infty}  +\int_{t_0}^t C(t-s)^{-\frac12}\|v(s)\|_{L^\infty}\|\nabla\rho(s)\|_{L^\infty}\,ds.
\end{align*}
This yields
\[
\sup_{t\in[t_0+\delta,T]}\|\nabla\rho(t)\|_{L^\infty} \le C_{\delta,T} \|\rho(t_0)\|_{L^\infty}  \exp \lt(C_{\delta,T}\|v\|_{L^1(t_0,T;L^\infty)}\rt).
\]

Next, using $\nabla\rho=\rho\,\phi$ and \eqref{eq:rho_mp}, we control the gradient at time $t_0$ by
\[
\|\nabla\rho(t_0)\|_{L^\infty} \le \|\rho_0\|_{L^\infty}\,\|\phi\|_{L^\infty(0,t_0;L^\infty)},
\]
where $\|\phi\|_{L^\infty(0,t_0;L^\infty)}<\infty$ follows from the local $H^s$-bound in Step 1. 

Finally, restarting \eqref{eq:duha} at $t_1:=t_0+\delta/2$ and applying the same smoothing argument to $\nabla\rho$, we obtain
\bq\label{eq:W1}
\sup_{t\in[t_0+\delta,T]}\|\nabla\rho(t)\|_{W^{1,\infty}} \le C_{\delta,T} \lt(\|\rho_0\|_{L^\infty} + \|\rho_0\|_{L^\infty}\|\phi\|_{L^\infty(0,t_0;L^\infty)}\rt) \exp \lt(C_{\delta,T}  \|v\|_{L^1(t_0,T;W^{1,\infty})} \rt).
\eq

\medskip
\noindent\textbf{Step 4. Conclusion.}
By \eqref{eq:rho_mp} and the identity
\[
\nabla\phi=\nabla\lt(\frac{\nabla\rho}{\rho}\rt)=\frac{\nabla^2\rho}{\rho}-\frac{\nabla\rho\otimes\nabla\rho}{\rho^2},
\]
we deduce for $t\in[t_0+\delta,T]$ that
\[
\|\nabla\phi(t)\|_{L^\infty} \le \frac{1}{c_0}\|\nabla^2\rho(t)\|_{L^\infty} +\frac{1}{c_0^2}\|\nabla\rho(t)\|_{L^\infty}^2.
\]
For $t\in[t_0+\delta,T]$, the right-hand side is uniformly bounded by \eqref{eq:W1}, and thus
\[
\int_{t_0+\delta}^{T}\|\nabla\phi(t)\|_{L^\infty}\,dt<\infty.
\]
On the other hand, on $[0,t_0+\delta]$ we already have $\phi\in L^\infty(0,t_0+\delta;H^s(\T^2))$, and this gives
\[
\int_0^{t_0+\delta}\|\nabla\phi(t)\|_{L^\infty}\,dt<\infty. 
\]
Hence, the continuation criterion \eqref{eq:cri} is satisfied on $[0,T]$. By Step 1, the local solution extends up to time $T$, and therefore $\rho\in X_{s,p}(T)$.
\end{proof}

%
%
%
%
%
%


\end{document}